\documentclass[11pt]{article}

\usepackage[T1]{fontenc}
\usepackage{lmodern}
\usepackage{amsmath,amssymb,amsthm,mathtools}
\usepackage{microtype}
\usepackage[margin=1in]{geometry}
\usepackage{enumitem}
\usepackage{xcolor}
\usepackage{framed}
\usepackage{aliascnt}
\usepackage[hidelinks]{hyperref}
\hypersetup{pdftitle={Exact Universality of Online Discrepancy}}
\usepackage[nameinlink,noabbrev]{cleveref}

\allowdisplaybreaks
\numberwithin{equation}{section}

\newtheorem{theorem}{Theorem}[section]

\newaliascnt{proposition}{theorem}
\newtheorem{proposition}[proposition]{Proposition}
\aliascntresetthe{proposition}

\newaliascnt{lemma}{theorem}
\newtheorem{lemma}[lemma]{Lemma}
\aliascntresetthe{lemma}

\newaliascnt{corollary}{theorem}
\newtheorem{corollary}[corollary]{Corollary}
\aliascntresetthe{corollary}

\theoremstyle{definition}
\newaliascnt{definition}{theorem}
\newtheorem{definition}[definition]{Definition}
\aliascntresetthe{definition}

\theoremstyle{remark}
\newaliascnt{remark}{theorem}

\aliascntresetthe{remark}

\crefname{theorem}{Theorem}{Theorems}
\Crefname{theorem}{Theorem}{Theorems}
\crefname{proposition}{Proposition}{Propositions}
\Crefname{proposition}{Proposition}{Propositions}
\crefname{lemma}{Lemma}{Lemmas}
\Crefname{lemma}{Lemma}{Lemmas}
\crefname{corollary}{Corollary}{Corollaries}
\Crefname{corollary}{Corollary}{Corollaries}
\crefname{definition}{Definition}{Definitions}
\Crefname{definition}{Definition}{Definitions}
\crefname{remark}{Remark}{Remarks}
\Crefname{remark}{Remark}{Remarks}
\crefname{section}{Section}{Sections}
\Crefname{section}{Section}{Sections}
\crefname{appendix}{Appendix}{Appendices}
\Crefname{appendix}{Appendix}{Appendices}

\newcounter{algorithm}
\crefname{algorithm}{Algorithm}{Algorithms}
\Crefname{algorithm}{Algorithm}{Algorithms}
\newenvironment{algorithmbox}[1]{%
  \begin{figure}[!ht]%
  \begin{minipage}{\linewidth}%
  \refstepcounter{algorithm}%
  \begin{framed}\noindent\textbf{Algorithm~\thealgorithm} (#1).\par\smallskip
}{\end{framed}\end{minipage}\end{figure}}

\newcommand{\R}{\mathbb{R}}
\newcommand{\E}{\mathbb{E}}
\newcommand{\Pp}{\mathbb{P}}
\newcommand{\1}{\mathbf{1}}
\newcommand{\cF}{\mathcal{F}}
\newcommand{\cG}{\mathcal{G}}
\newcommand{\cH}{\mathcal{H}}
\newcommand{\cM}{\mathcal{M}}
\newcommand{\norm}[1]{\lVert #1\rVert}
\newcommand{\esssup}{\operatorname*{ess\,sup}}
\newcommand{\dist}{\operatorname{dist}}
\newcommand{\aonmu}{\alpha_{\mathrm{on}}^{\mathrm{SBP},\mu}}
\newcommand{\cstar}{c_{*}}
\newcommand{\ca}{c_{\alpha}}
\newcommand{\eps}{\varepsilon}
\newcommand{\Law}{\operatorname{Law}}
\newcommand{\Var}{\operatorname{Var}}
\newcommand{\dd}{\,\mathrm{d}}

\title{Exact Universality of Online Discrepancy}
\author{%
  Sunghyeon Jo\\
  {\small Georgia Institute of Technology and QED Audit}\\
  {\small\texttt{sjo65@gatech.edu}}
  \and
  Taekyun Lee\\
  {\small The University of Texas at Austin}\\
  {\small\texttt{taekyun@utexas.edu}}
}
\date{\today}

\begin{document}
\maketitle

\begin{abstract}
We study online vector balancing with $N$ random vectors in $\R^M$
revealed sequentially, where each vector must be assigned an irrevocable
sign upon arrival.  The goal is to minimize the expected $\ell^\infty$ norm
of the final signed sum.

For i.i.d.\ entries with mean zero, variance one, and a finite fourth
moment, we prove that, as $M/N\to\alpha\in(0,\infty)$, the optimal value
divided by $\sqrt N$ converges to a limit $R_\alpha$ independent of the
entry distribution.  This limit is the stochastic control value identified
for Gaussian inputs by Fiedler, Jackson, Lacker, and Niles-Weed.  In
particular, it determines the exact asymptotic optimum for Rademacher inputs.

For every $\kappa>R_\alpha$, we construct a randomized online algorithm
whose final signed sum has $\ell^\infty$ norm at most $\kappa\sqrt N$ with
high probability; for $\kappa<R_\alpha$, every online algorithm has vanishing
success probability.  Consequently, the online threshold of the symmetric
binary perceptron is universal at every positive margin.  The main step is
a coupling that transfers Brownian controls to non-Gaussian inputs, while
truncation controls rare large entries.
\end{abstract}

\section{Introduction}
\label{sec:introduction}

In online vector balancing, random vectors $\xi_1,\ldots,\xi_N\in\R^M$
arrive one at a time.  After seeing $\xi_k$, an online algorithm chooses
$\eps_k\in\{-1,+1\}$ before the next vector is revealed.  Its normalized terminal
discrepancy is
\[
D_{M,N}:=\left\|\frac1{\sqrt N}\sum_{k=1}^N\eps_k\xi_k\right\|_\infty.
\]
We assume that all entries are independent with a common law $\mu$ of mean
zero and variance one, and consider $M/N\to\alpha\in(0,\infty)$.
Randomized algorithms may use a seed independent of the entries.

For Rademacher entries, taking the values $-1$ and $+1$ with equal
probability, this is a classical problem in online discrepancy.
Bansal and Spencer \cite[Theorem~3.4]{BansalSpencer20} give an online
algorithm with unnormalized discrepancy $O(\sqrt M)$ with high probability
at a prescribed terminal time.  This gives $D_{M,N}=O(1)$ in the proportional regime,
where Gamarnik, K{\i}z{\i}lda\u{g}, Perkins, and Xu
\cite[Theorem~10]{GamarnikKizildagPerkinsXu23} establish a lower bound of
constant order (for $M>N$, by restriction to any $N$ rows).
These results leave open the exact limiting optimum.
Recently, Fiedler, Jackson, Lacker, and Niles-Weed
\cite{FiedlerJacksonLackerNilesWeed26} determined that optimum for
Gaussian entries by identifying a one-dimensional Brownian control problem.
They also proved the same lower bound for every centered entry
distribution with variance one and a finite fourth moment.  Their upper bound uses
Gaussian structure, and their Section~1.6 asks whether a matching
non-Gaussian upper bound holds.

We resolve this question under the same assumption on the fourth moment.
In particular, the exact limiting optimum for Rademacher inputs equals
the Gaussian optimum.

\subsection{Our results}
\label{sec:main-results}

Let $B$ be a standard Brownian motion.  For $T>0$, define
\[
V_T^\infty:=\inf_u
\left\|B_T+\int_0^T u_t\,\dd t\right\|_{L^\infty},
\]
where $\|Y\|_{L^\infty}$ denotes the essential supremum of $|Y|$.
The infimum is over controls progressively measurable in the natural
Brownian filtration and satisfying $\E u_t^2\le2/\pi$ for almost every
$t\in[0,T]$.  This is the smallest radius of a centered interval containing
the terminal controlled process almost surely.  Following
\cite{FiedlerJacksonLackerNilesWeed26}, set
\[
R_\alpha:=\sqrt{\alpha}\,V_{1/\alpha}^\infty.
\]
\Cref{sec:benchmark} recalls the Brownian scaling and properties of this value.
We write $(x)_+:=\max\{x,0\}$ for the positive part.

\begin{theorem}[Universality of online discrepancy]
\label{thm:main}
Fix $\alpha>0$ and let $M/N\to\alpha$.  If $\mu$ has mean zero, variance
one, and a finite fourth moment, then
\[
\inf_{\eps\,\mathrm{online}}\E_\mu D_{M,N}\longrightarrow R_\alpha.
\]
Moreover:
\begin{enumerate}[label=(\roman*),leftmargin=2.4em]
\item For every fixed $\kappa>R_\alpha$ and $K<\infty$, there is a
randomized online algorithm, depending only on $(\alpha,\kappa,K)$, such
that for every centered law $\mu$ with variance one and $\E_\mu|\xi|^4\le K$,
\[
\Pp_\mu(D_{M,N}>\kappa)+\E_\mu(D_{M,N}-\kappa)_+
\longrightarrow0.
\]
\item For every fixed $\kappa<R_\alpha$,
\[
\sup_{\eps\,\mathrm{online}}\Pp_\mu(D_{M,N}\le\kappa)
\longrightarrow0.
\]
\end{enumerate}
\end{theorem}

The principal application is to the symmetric binary perceptron (SBP),
which asks for signs satisfying
\[
\left|\sum_{k=1}^N\eps_k\xi_{k,i}\right|\le\kappa\sqrt N
\qquad\text{for every }i\le M,
\]
equivalently, $D_{M,N}\le\kappa$.  Offline one chooses all signs after
seeing the entire input matrix; online one fixes $\eps_k$ when column $k$
arrives.  Let $\aonmu(\kappa)$ be the supremum of the densities $\alpha$
at which online algorithms satisfy all constraints with probability
tending to one, as defined precisely in
\eqref{eq:SBP-threshold-definition}.

\begin{corollary}[The online SBP threshold]
\label{cor:sbp}
For every $\kappa>0$ and every $\mu$ as in \cref{thm:main},
$\aonmu(\kappa)$ is independent of $\mu$ and is the unique $\alpha>0$
satisfying
\[
\kappa=R_\alpha=\sqrt{\alpha}\,V^\infty_{1/\alpha}.
\]
Every fixed density below this value is achievable by an online
algorithm.  At every fixed density above it, the supremal success
probability tends to zero.
\end{corollary}

Writing $\ell_\star:=\lim_{T\to\infty}V_T^\infty$, we obtain
\[
\aonmu(\kappa)\sim \ell_\star^{-2}\kappa^2\qquad(\kappa\downarrow0).
\]
Thus the coefficient is the same for every law in \cref{thm:main}.
\Cref{sec:sbp} derives the asymptotics and compares the online and offline
capacities.  For Gaussian and Rademacher inputs, their ratio tends to zero
as $\kappa\downarrow0$.

Our upper bound is quantitative.  Write
\[
\mathfrak t_\mu(H):=\E_\mu[|\xi|^4\1_{\{|\xi|>H\}}].
\]

\begin{theorem}[Quantitative upper bound]
\label{thm:quantitative}
Fix $\alpha>0$, $\kappa>R_\alpha$, $K<\infty$, and $A>0$.
There is a randomized online algorithm, depending only on these
parameters, such that every centered law $\mu$ with variance one and
$\E_\mu|\xi|^4\le K$ satisfies
\begin{equation}
\Pp_\mu(D_{M,N}>\kappa)+\E_\mu(D_{M,N}-\kappa)_+
\le C_A\left[N^{-A}+\mathfrak t_\mu(N^{1/3})\right].
\label{eq:main-tail-modulus}
\end{equation}
The bound holds uniformly for all sufficiently large $N$ and all $M/N$
in a sufficiently small fixed neighborhood of $\alpha$.
The constant, the neighborhood, and the lower bound on $N$ depend only on
$(\alpha,\kappa,K,A)$.
\end{theorem}

The tail term tends to zero for each fixed law, and the convergence
is uniform over classes with uniformly integrable fourth moments.
All quantitative bounds are for a fixed positive gap $\kappa-R_\alpha$.

\begin{corollary}[Bounds under higher moments]
\label{cor:higher-moment-rate}
Fix $\alpha>0$, $q>4$, $K<\infty$, and $\kappa>R_\alpha$, and let
$M/N\to\alpha$.  One randomized online algorithm satisfies
\begin{equation}
\sup_{\substack{\E_\mu\xi=0,\ \E_\mu\xi^2=1\\
                \E_\mu|\xi|^q\le K}}
\left\{\Pp_\mu(D_{M,N}>\kappa)+\E_\mu(D_{M,N}-\kappa)_+\right\}
=O\!\left(N^{-(q-4)/2}\right).
\label{eq:higher-moment-rate}
\end{equation}
\end{corollary}

This rate follows by applying Rosenthal's inequality with exponent $q$
to the discarded entries.  For independent entries with different
distributions, \cref{thm:independent-array} proves the same limit under
uniformly integrable fourth moments.  Its lower bound holds under a
common bound on fourth moments.

For fixed precomputed control data, the algorithms in
\cref{thm:quantitative,cor:higher-moment-rate} use $O(M)$ arithmetic operations
and scalar feedback evaluations per arrival, with $O(M)$ working memory.
\Cref{sec:online-algorithm} specifies this evaluation model.

The Gaussian value and the expected lower bound are due to
\cite{FiedlerJacksonLackerNilesWeed26}.  \Cref{sec:lower} also deduces
part (ii) and Gaussian feasibility from their results.

\subsection{Overview of the proof}
\label{sec:overview}

Extending the Gaussian construction requires estimates for both the
joint distribution of the signed vector and the largest terminal
coordinate.  The chosen sign depends on the entire incoming vector and
couples its coordinates.  A central limit theorem identifies the
approximate drift; following the controlled Brownian paths also requires
a coupling of the signed vector.  Its errors must remain small in all
$M=\Theta(N)$ coordinates so that every terminal constraint is satisfied.

\paragraph{Coupling the signed vector.}
Suppose the desired nonzero drift $r$ has been computed from the past.
For the truncated vector $\zeta$, choose $\eps$ with conditional mean
$\psi_\rho(\langle v,\zeta\rangle)$, where
$v=r/\|r\|_2$, $\rho=\|r\|_2/\sqrt N$, and $\psi_\rho$ is a smooth odd
function calibrated so that $\E[G\psi_\rho(G)]=\rho$ for
$G\sim N(0,1)$.  For sufficiently small
$\|v\|_\infty$, normal approximation gives
$\E(\eps\zeta_i\mid\text{past})=r_i/\sqrt N+e_i$ with a small error.
A small independent Rademacher perturbation gives the needed bound on
$\|v\|_\infty$ with high probability, subject to drift bounds checked by
the algorithm.

Conditionally on $\eps$, the unsigned vector $\zeta$ has a bounded
density relative to a product measure, including for nonsymmetric inputs.
Successive maximal couplings of the coordinates produce a vector $Z$
with independent coordinates, independent of $\eps$.  Coupling the slightly biased
sign $\eps$ to a fair sign $S$ then gives
\[
\eps\zeta=SZ+\frac r{\sqrt N}+e+D,
\]
where $D$ is centered conditionally on the past and has small coordinate
variances.  Conditionally on every past, $S$ is a fair Rademacher sign,
$Z\sim\nu_N^{\otimes M}$, and $S$ and $Z$ are independent; this conditional
law does not depend on the past.  The coordinates of $SZ$ are independent
\emph{conditionally on $S$}.
The algorithm computes $\eps$ from the observed vector and an independent
uniform random variable; $(S,Z)$ is constructed on an enlarged probability
space for the proof (\cref{lem:product-innovation}).

The fixed conditional law makes the pairs $(S_k,Z_k)$ independent across
arrivals.  Iterating the decomposition gives the prescribed drift, a
predictable error, and a martingale, in addition to the independent
increments $S_kZ_k$.  We bound the accumulated errors in the maximum norm,
estimating the martingale in the filtration generated by the construction.
We then condition on the auxiliary signs and apply a strong Gaussian
approximation to each row.

\paragraph{Truncation and the fourth moment.}
The algorithm computes its sign using the truncated entry
$\zeta=\xi\1_{\{|\xi|\le N^{1/3}\}}$ and applies that sign to $\xi$.
The coupling estimates include the mean and variance errors introduced
by truncation.
The truncated entries have moments of every order.  Tracking their growth
with $N$, for every fixed $A>0$ we can choose parameters so that the
algorithm on truncated inputs has failure probability $O(N^{-A})$.

Let $D^{\mathrm{tr}}_{M,N}$ denote the discrepancy of the truncated signed
sum.  For every choice of signs, the discarded entries satisfy
\[
|D_{M,N}-D^{\mathrm{tr}}_{M,N}|
\le \max_{i\le M}\frac1{\sqrt N}\sum_{k=1}^N
|\xi_{k,i}|\1_{\{|\xi_{k,i}|>N^{1/3}\}}.
\]
A fourth-moment estimate for the centered row sums, together with a fixed
gap between the target for the truncated process and $\kappa$, proves
\eqref{eq:main-tail-modulus} (\cref{sec:tail-transfer}).

\paragraph{The online algorithm.}
Fix a smaller target $a\in(R_\alpha,\kappa)$ and reserve a short final
correction block.  Before the split, the algorithm follows a bounded simple
predictable approximation of a Brownian control.  Its coefficients depend on
finitely many stored observations, and subtracting the prescribed drift from
the truncated signed sum supplies approximations to the Brownian observations
used by this control.  At the split time, the distances to $[-a,a]$ have a
small empirical moment.

During the final block, a correction drift determined by the split-time state
is held fixed, while a F\"ollmer drift confines the new fluctuation to a small
interval.  The feedback is dissipative, so the Euler scheme is stable under
the accumulated coupling errors.  At each arrival the drift is computed
before the vector is seen, perturbed by a small random vector, and checked
against fixed bounds.  If a check fails, a bounded auxiliary drift is used
thereafter.  \Cref{alg:online} specifies every update.  We prove that all
checks pass with high probability.  The auxiliary drift keeps the coupling
estimates valid on exceptional trajectories, allowing us to bound the
expected loss.

The correction and the F\"ollmer drift follow
\cite[Section~4.1]{FiedlerJacksonLackerNilesWeed26} and
\cite[Section~4 and Appendix~C]{JoLee26}.
The Gaussian approximation is carried out separately before and after the
split so that the Brownian increments used in the final block are independent
of the pre-split history.

\paragraph{Lower bounds.}
The compactness argument of \cite{FiedlerJacksonLackerNilesWeed26}
retains the random empirical measure of the drift and martingale paths
of the coordinates.  Almost every limiting measure is an admissible Brownian
control law.  Measures supported at terminal time in $[-\kappa,\kappa]$
form a closed set disjoint from the admissible limits when
$\kappa<R_\alpha$.  Applying the Portmanteau theorem to the law of the
random empirical measure gives vanishing probability of simultaneous
feasibility.
\Cref{sec:lower} gives the argument; \cref{sec:independent-array} proves
the additional estimates needed when the entry distributions differ.

\subsection{Related work}
\label{sec:prior-work}

The stochastic online discrepancy literature includes the Rademacher
result of Bansal and Spencer \cite{BansalSpencer20} and bounds for more
general arrival distributions
\cite{BansalJiangSinglaSinha20,BansalEtAl21}.
Wengiel \cite{Wengiel26} gives a polynomial-time algorithm with unnormalized terminal
discrepancy $O(\sigma^8\sqrt M)$ for independent symmetric sub-Gaussian
coordinates, uniformly over the horizon, when $M$ is sufficiently large
relative to the sub-Gaussian norm bound $\sigma$.
That result addresses arbitrary prescribed horizons and efficient
computation; our result identifies the exact constant in the proportional
regime for distributions with a finite fourth moment.

For vectors chosen by an oblivious adversary, random walks control every
signed prefix
\cite{AlweissLiuSawhney21,LiuSahSawhney22,KulkarniReisRothvoss24,AdenAli26}.
Other models include sparse binary columns \cite{AltschulerTikhomirov25},
online Beck--Fiala \cite{AltschulerTikhomirov26}, and vectors with small
coordinates \cite{Hmadi26}.  Continuous methods also describe adversarial
batched balancing \cite{GuillenKobzar26} and confinement when steps may
be discarded \cite{SmirnovVershynin26}.

Discrepancy of random matrices and offline perceptron capacity have been
studied in several regimes
\cite{TurnerMekaRigollet20,AltschulerNilesWeed22}.
The exact Gaussian and binary SBP capacities are known
\cite{AubinPerkinsZdeborova19,PerkinsXu21,AbbeLiSly21}; we recall them in
\cref{sec:sbp}.  Typical freezing is compatible with rare well-connected
clusters \cite{AbbeLiSly22}.
Nakajima and Sun \cite{NakajimaSun23} prove universality of the free energy
and a sharp threshold sequence for sub-Gaussian disorder.
The restrictions on stable algorithms studied in
\cite{GamarnikKizildagPerkinsXu22,GamarnikKizildagPerkinsXu23} differ
from irrevocable online decisions.

Huang, Sellke, and Sun \cite{HuangSellkeSun26} use stochastic control to
characterize the empirical distributions of projections attainable by
algorithms on Gaussian matrices with a dimension-independent Lipschitz
constant.
Their Section~6.7 announces a companion treatment of exact feasibility.
The relevant distinction here is the information available online:
our lower bound covers every online algorithm, and our upper
bound applies to non-Gaussian product distributions.

The companion paper \cite{JoLee26} determines the Gaussian online
threshold for the asymmetric binary perceptron.  Its constraint is
one-sided.  We use its formulation of simple predictable controls and
recall its F\"ollmer and Euler estimates, with precise references in
\cref{sec:controls-slack,sec:confinement,app:Follmer}.
For the interval constraint here, conditional projection decreases the
moment of the distance to the target interval, by convexity.

\paragraph{Organization.}
\Cref{sec:benchmark} recalls the Brownian control problem.
\Cref{sec:coupling} proves the coupling estimates, and
\cref{sec:algorithm} constructs and analyzes the online algorithm.
\Cref{sec:lower} proves the lower bounds.
\Cref{sec:extensions} treats higher moments and independent entries with
different distributions, and \cref{sec:sbp} gives the perceptron threshold.
The appendices contain Brownian control projection and approximation,
measurable coupling kernels, F\"ollmer and Euler estimates, and quantitative
estimates.

\section{The Brownian control problem}
\label{sec:benchmark}

\subsection{Online algorithms and normalization}
\label{sec:model}

Throughout the paper, $M=M_N$ satisfies $M/N\to\alpha\in(0,\infty)$.  Let
$\xi_k=(\xi_{k,1},\ldots,\xi_{k,M})$, $1\le k\le N$, be independent random
vectors whose coordinates are i.i.d.\ with common law $\mu$.
A randomized online algorithm carries an external seed $\mathsf U$,
independent of the input, and chooses
\[
\eps_k=\mathsf A_{k,M,N}(\mathsf U,\xi_1,\ldots,\xi_k)\in\{-1,+1\},
\qquad 1\le k\le N.
\]
Its normalized state is
\begin{equation}
X_k:=\frac1{\sqrt N}\sum_{j=1}^k\eps_j\xi_j,
\qquad 0\le k\le N,
\qquad
D_{M,N}:=\norm{X_N}_\infty,
\label{eq:partial-sums}
\end{equation}
and its optimal expected discrepancy is
$\inf_{\eps\,\mathrm{online}}\E_\mu D_{M,N}$.
All probabilities and expectations include the algorithm's independent
randomness unless we explicitly condition on it.

For $q\ge4$ and $K<\infty$, write $\cM_q(K)$ for the class of laws $\mu$ on
$\R$ satisfying
\begin{equation}
\E_\mu\xi=0,
\qquad \E_\mu\xi^2=1,
\qquad \E_\mu|\xi|^q\le K.
\label{eq:moment-class}
\end{equation}

\subsection{The Brownian control value}
\label{sec:brownian-value}

Let $B$ be a standard Brownian motion with completed natural filtration
$\cF^B=(\cF^B_t)_{t\ge0}$.  For $T>0$, \cite{FiedlerJacksonLackerNilesWeed26} defines
\begin{equation}
V_T^\infty
:=
\inf\left\{
\left\|B_T+\int_0^T u_t\,\dd t\right\|_{L^\infty}:
u\ \cF^B\text{-progressive},\ \
\norm{u_t}_{L^2}\le\sqrt{\tfrac2\pi}
\ \text{for a.e.\ }t
\right\},
\label{eq:VT-definition}
\end{equation}
where $\norm{Y}_{L^\infty}$ is the essential supremum of $|Y|$.
For $c\ge0$, we also use the value at time one with drift bound $c$:
\begin{equation}
\mathcal R(c):=\inf\left\{
\left\|B_1+\int_0^1u_t\,\dd t\right\|_{L^\infty}:
u\ \cF^B\text{-progressive},\ \
\norm{u_t}_{L^2}\le c\ \text{a.e.}
\right\}.
\label{eq:R-of-c}
\end{equation}
Set
\begin{equation}
\cstar:=\E|G|=\sqrt{\frac2\pi},\qquad G\sim N(0,1),
\qquad
\ca:=\frac{\cstar}{\sqrt\alpha}=\sqrt{\frac{2}{\pi\alpha}}.
\label{eq:steering-constants}
\end{equation}
For $c>0$, Brownian scaling gives
\begin{equation}
\mathcal R(c)=\frac{\cstar}{c}V^\infty_{(c/\cstar)^2},
\qquad\text{and in particular}\qquad
R_\alpha=\sqrt\alpha\,V^\infty_{1/\alpha}=\mathcal R(\ca).
\label{eq:Ralpha-cap}
\end{equation}
At $c=0$, the only admissible drift is zero $\dd t\otimes\Pp$-almost
everywhere, and $\mathcal R(0)=\norm{B_1}_{L^\infty}=\infty$.

We recall the following properties of the Brownian value.

\begin{proposition}
\label{prop:brownian-facts}
The infimum in \eqref{eq:R-of-c} is attained.  The map
$c\mapsto\mathcal R(c)$ is continuous and nonincreasing on $(0,\infty)$, the
map $T\mapsto V_T^\infty$ is nondecreasing, and
$\ell_\star:=\lim_{T\to\infty}V_T^\infty$ is finite.
\end{proposition}

\begin{proof}
These follow from \cite[Lemma~2.5 and
Proposition~2.8]{FiedlerJacksonLackerNilesWeed26} after the Brownian scaling
\eqref{eq:Ralpha-cap}.  The compactness argument first gives an optimal
control adapted to a filtration in which $B$ is Brownian, possibly larger
than its natural filtration.  Projection onto the natural Brownian
filtration (\cref{lem:Brownian-projection}) preserves the terminal bound
and the pointwise $L^2$ constraint.  The projected control can then be
represented on the fixed canonical Wiener space.
Finiteness of $\ell_\star$ is
\cite[Section~6]{FiedlerJacksonLackerNilesWeed26}.
\end{proof}

\section{A coupling for non-Gaussian signings}
\label{sec:coupling}

The Gaussian upper bound of \cite{FiedlerJacksonLackerNilesWeed26}
represents the signed vector using independent Gaussian increments.
For the analysis, we couple the incoming vector and its chosen sign to
a vector with independent coordinates and an auxiliary fair sign.
Truncation allows this construction for every law in $\cM_4(K)$.

Throughout the upper construction, assume $\mu\in\cM_4(K)$ and put
\begin{equation}
H_N=N^{1/3},\qquad
\zeta=\xi\1_{\{|\xi|\le H_N\}},\qquad
\nu_N=\Law(\zeta),\qquad
m_N=\E\zeta,\quad \sigma_N^2=\Var(\zeta).
\label{eq:truncation-definition}
\end{equation}
The algorithm updates its state using $\zeta$.  The moments $m_N$ and
$\sigma_N$ enter only the analysis.
Uniformly over $\cM_4(K)$,
\begin{equation}
|m_N|\le K N^{-1},\qquad
|\sigma_N^2-1|\le K N^{-2/3}+K^2N^{-2},\qquad
\E|\zeta|^{s}\le K N^{(s-4)/3}\quad(s>4).
\label{eq:truncation-moments}
\end{equation}
Indeed, centering the original law bounds $|m_N|$ by
$\E(|\xi|;|\xi|>H_N)$, and the lost second moment is at most
$KH_N^{-2}$.  Also $\E|\zeta|^j\le 1+K$ for $j\le4$.
We reserve $q$ for a moment assumption on the original entries.  The
auxiliary exponent $s>4$ is used both for moments of truncated entries
and for empirical moments of the drift.  It can be chosen as large as
needed without strengthening the assumption on $\mu$.

\subsection{Realizing a prescribed mean drift}
\label{sec:one-column}

Fix a gap below the maximum drift magnitude,
\begin{equation}
0<\delta<\cstar.
\label{eq:steering-slack}
\end{equation}
Choose $L=L_\delta<\infty$ so that
\begin{equation}
m_\delta:=\E[G\tanh(LG)]>\cstar-\frac\delta2.
\label{eq:mdelta}
\end{equation}
For $0\le\rho\le\cstar-\delta$, define
\begin{equation}
\psi_\rho(x):=\frac{\rho}{m_\delta}\tanh(Lx).
\label{eq:psi-definition}
\end{equation}
Then $\psi_\rho$ is odd and
\begin{equation}
\norm{\psi_\rho}_\infty\le
a_\delta:=\frac{\cstar-\delta}{\cstar-\delta/2}<1,
\qquad
\E[G\psi_\rho(G)]=\rho,
\label{eq:psi-calibration}
\end{equation}
and, for each fixed integer $j\ge1$,
\begin{equation}
\norm{\psi_\rho^{(j)}}_\infty\le C_{\delta,j}\rho.
\label{eq:psi-derivatives}
\end{equation}

\begin{lemma}
\label{lem:low-influence-response}
Let $\zeta_1,\ldots,\zeta_M$ be independent copies of a law $\nu$ with
mean $m$, variance $\sigma^2\in[1/2,2]$, and bounded third absolute moment.
Write $\zeta=(\zeta_1,\ldots,\zeta_M)$, suppose that
$\norm v_2=1$ and $\norm v_\infty\le1/2$, and put $d=m\sum_jv_j$.
For $0\le\rho\le\cstar-\delta$,
\begin{align}
\E[\zeta_i\psi_\rho(\langle v,\zeta\rangle)]&=\rho v_i+e_i,\nonumber\\
|e_i|&\le C\rho\left[
 |v_i|(\norm v_\infty+|\sigma^2-1|+|d|)+|m|\right].
\label{eq:response-bound}
\end{align}
Moreover,
\begin{equation}
t:=\E\psi_\rho(\langle v,\zeta\rangle),\qquad
|t|\le C\rho\left(\sum_j|v_j|^3+|d|\right)
\le C\rho(\norm v_\infty+|d|).
\label{eq:sign-bias}
\end{equation}
Constants depend only on $\delta$ and the third-moment bound.
\end{lemma}

\begin{proof}
Write $\zeta_i^\circ=\zeta_i-m$ and
$A_i=\sum_{j\ne i}v_j\zeta_j^\circ$.  A first-order Taylor expansion,
using $\E\zeta_i^\circ=0$, gives
\begin{align}
\E[\zeta_i^\circ\psi_\rho(A_i+d+v_i\zeta_i^\circ)]
 &=v_i\sigma^2\E\psi_\rho'(A_i+d)+R_i,
\label{eq:Taylor-response}\\
|R_i|&\le C\rho v_i^2\E|\zeta_i^\circ|^3.
\label{eq:Taylor-remainder}
\end{align}
For $p\ge1$, we write $\mathbf d_p(\mu,\nu)$ for the $p$-Wasserstein
distance: the infimum of $\|X-Y\|_{L^p}$ over couplings with laws
$\mu$ and $\nu$.  The Wasserstein Berry--Esseen inequality for independent
weighted sums
\cite[Theorem~3.1]{ChenGoldsteinShao11} implies
\begin{equation}
\mathbf d_1\!\left(\Law(A_i),
       N(0,\sigma^2(1-v_i^2))\right)
 \le C\sum_{j\ne i}|v_j|^3\le C\norm v_\infty.
\label{eq:Wasserstein-BE}
\end{equation}
The variance $\sigma^2(1-v_i^2)\ge3/8$ keeps the normalizing factor
in the Berry--Esseen bound uniformly bounded.  By Lipschitzness of
$\psi_\rho'$,
\begin{equation}
\left|\E\psi_\rho'(A_i+d)-\E\psi_\rho'(G)\right|
 \le C\rho(\norm v_\infty+|\sigma^2-1|+|d|).
\label{eq:derivative-CLT}
\end{equation}
Gaussian integration by parts gives
\begin{equation}
\E\psi_\rho'(G)=\E[G\psi_\rho(G)]=\rho.
\label{eq:Gaussian-IBP}
\end{equation}
The remaining mean contribution is
$m\,\E\psi_\rho(\langle v,\zeta\rangle)$, bounded by $C\rho|m|$.
Combining these estimates proves \eqref{eq:response-bound}.
For \eqref{eq:sign-bias}, compare $\sum_jv_j\zeta_j^\circ$ with
$\sigma G$ in Wasserstein distance and shift by $d$.
Oddness gives $\E\psi_\rho(\sigma G)=0$ for every $\sigma$, yielding
\eqref{eq:sign-bias} directly.
\end{proof}

For a drift $r\ne0$ chosen before the current column is observed,
condition on the available past and put
\begin{equation}
\rho=\frac{\norm r_2}{\sqrt N},\qquad v=\frac r{\norm r_2}.
\label{eq:rho-v-request}
\end{equation}
After observing the truncated vector $\zeta$, draw an independent
$U'\sim\operatorname{Unif}[-1,1]$ and choose
\begin{equation}
\eps=
\begin{cases}
+1,&U'\le\psi_\rho(\langle v,\zeta\rangle),\\
-1,&U'>\psi_\rho(\langle v,\zeta\rangle).
\end{cases}
\label{eq:sign-selection}
\end{equation}
For $r=0$, choose an independent fair sign.  For the truncated law,
the preceding lemma yields, with expectations conditional on that past,
\begin{equation}
\E[\eps\zeta_i]=\frac{r_i}{\sqrt N}+e_i,\qquad
|e_i|\le C\left[
 \frac{|r_i|}{\sqrt N}(\norm v_\infty+N^{-2/3}+\sqrt M N^{-1})
 +\rho N^{-1}\right].
\label{eq:one-column-drift-error}
\end{equation}

\subsection{An auxiliary sign for nonsymmetric inputs}
\label{sec:signed-law}

Conditioning the unsigned vector on the chosen sign gives a bounded
density relative to $\nu^{\otimes M}$, including when $\nu$ is nonsymmetric.

\begin{lemma}
\label{lem:exact-tilt}
Under \eqref{eq:sign-selection}, with
$t=\E\psi_\rho(\langle v,\zeta\rangle)$, for $\sigma\in\{-1,1\}$,
\begin{equation}
\Pp(\eps=\sigma)=\frac{1+\sigma t}{2},\qquad
\frac{\dd\Law(\zeta\mid\eps=\sigma)}{\dd\nu^{\otimes M}}(x)
 =\frac{1+\sigma\psi_\rho(\langle v,x\rangle)}{1+\sigma t}.
\label{eq:product-tilt-density}
\end{equation}
In particular the conditional density lies between
\begin{equation}
\frac{1-a_\delta}{1+a_\delta}
\quad\text{and}\quad
\frac{1+a_\delta}{1-a_\delta}.
\label{eq:tilt-bounds}
\end{equation}
\end{lemma}

\begin{proof}
The joint law is
$\Pp(\eps=\sigma,\zeta\in\dd x)
 =\tfrac12(1+\sigma\psi_\rho(\langle v,x\rangle))\nu^{\otimes M}(\dd x)$.
Integrating and conditioning give the claims.
\end{proof}

\begin{lemma}
\label{lem:product-innovation}
Under the hypotheses of \cref{lem:low-influence-response}, let $\eps$
be chosen by \eqref{eq:sign-selection}.  Couple the actual pair
$(\eps,\zeta)$ to a fair Rademacher sign $S$ and a vector
$Z\sim\nu^{\otimes M}$ such that $Z$ is independent of $(\eps,S)$.
Writing $Y=\eps\zeta$, one can arrange, for every $s\ge3$
with $\E|\zeta_1|^s<\infty$,
\begin{align}
\E|Y_i-SZ_i|^2&\le C(\rho|v_i|+|t|),
\label{eq:YZ-second}\\
\E|Y_i-SZ_i|^{s}&\le C_{s}\E|\zeta_1|^{s}.
\label{eq:YZ-q}
\end{align}
For
\begin{equation}
D=Y-SZ-\E Y,
\label{eq:D-definition}
\end{equation}
we have
\begin{equation}
\E D_i=0,\qquad
\E D_i^2\le C\rho(|v_i|+\norm v_\infty+\sqrt M|m|),\qquad
\E|D_i|^{s}\le C_{s}\E|\zeta_1|^{s}.
\label{eq:D-moment-bounds}
\end{equation}
The constants in the second-moment bounds depend only on $\delta$ and
a third-moment bound; the $s$th moment enters linearly as displayed.
For fixed $\nu$, the conditional law of $(S,Z)$ given $(\eps,\zeta)$
is measurable in the parameters $(\rho,v,t)$ and the observed pair.
\end{lemma}

\begin{proof}
Maximally couple $\eps$ to a fair sign $S$.  Given the observed pair,
sample $S$ from the conditional law
\[
\Pp(S=-\eps\mid\eps,\zeta)
 =\frac{(\eps t)_+}{1+\eps t}.
\]
This conditional probability depends only on $\eps$, and
$\Pp(S\ne\eps)=|t|/2$.  Conditioning further on $S$ therefore preserves
\eqref{eq:product-tilt-density}.

Fix $\eps=\sigma$ and $S$.  Put $h_\sigma(x)=1+\sigma\psi_\rho(\langle v,x\rangle)$ and
\begin{equation}
H_i(x_1,\ldots,x_i)
 =\int h_\sigma(x_1,\ldots,x_i,y_{i+1},\ldots,y_M)
      \prod_{j>i}\nu(\dd y_j).
\label{eq:Hi-definition}
\end{equation}
Here $H_0=1+\sigma t$ and $H_M=h_\sigma$.  Write
$x_{<i}=(x_1,\ldots,x_{i-1})$.  The conditional density of $\zeta_i$
with respect to $\nu$, given $\zeta_{<i}=x_{<i}$, is
\begin{equation}
g_i(x)=\frac{H_i(x_{<i},x)}{H_{i-1}(x_{<i})}.
\label{eq:gi-density}
\end{equation}
The normalization $1+\sigma t$ cancels; $H_{i-1}\ge1-a_\delta$.
Comparing $x$ with an independent $\zeta_i'$ gives
\begin{equation}
|g_i(x)-1|\le C_\delta\rho|v_i|(|x|+\E|\zeta_1|).
\label{eq:gi-minus-one}
\end{equation}
Apply maximal coupling to $g_i\,\dd\nu$ and $\dd\nu$ successively.
Each $Z_i$ is sampled conditionally on the observed prefix
$(\zeta_1,\ldots,\zeta_i)$, as described in \cref{app:coupling-kernels}.
Its conditional law given the preceding $\zeta$ and $Z$ coordinates and
$(\eps,S)$ is $\nu$.  Induction therefore gives the asserted marginal law and independence.
The residual measures give
\begin{equation}
\E[|\zeta_i-Z_i|^2\mid\eps,S,\zeta_{<i},Z_{<i}]
 \le C\int(1+x^2)|g_i(x)-1|\,\nu(\dd x)
 \le C\rho|v_i|.
\label{eq:maximal-second-proof}
\end{equation}
Using
$Y_i-SZ_i=\eps(\zeta_i-Z_i)+(\eps-S)Z_i$
and independence of $Z$ from $(\eps,S)$ proves
\[
\E|Y_i-SZ_i|^2
 \le C\rho|v_i|+8\E Z_i^2\,\Pp(S\ne\eps)
 \le C(\rho|v_i|+|t|).
\]
The $s$th-moment bound follows from the marginals and
$|x-y|^{s}\le2^{s-1}(|x|^{s}+|y|^{s})$.
Since $\E[SZ_i]=0$, subtracting the mean proves
\eqref{eq:D-moment-bounds}, using \eqref{eq:sign-bias}.
The kernels are measurable by \cref{app:coupling-kernels}.
\end{proof}

Combining the two estimates gives the exact decomposition
\begin{equation}
\eps\zeta=SZ+\frac r{\sqrt N}+e+D.
\label{eq:one-step-decomposition}
\end{equation}
For zero drift, take the independent fair sign $\eps$ as $S$ and set
$Z=\zeta$; then $D=e=0$.

\subsection{Delocalizing the prescribed drift}
\label{sec:delocalization}

A drift concentrated on a few coordinates can give a large error in
\eqref{eq:one-column-drift-error}, especially during the final correction.
Before observing each vector, we perturb the predictable base drift $b_k$:
draw a fresh vector $\eta_k\in\{\pm1\}^M$ with independent Rademacher
coordinates, independently of the past, and form the candidate
\begin{equation}
\widetilde r_k:=b_k+\gamma\eta_k,
\qquad \gamma>0.
\label{eq:perturbed-drift}
\end{equation}

\begin{lemma}
\label{lem:delocalization}
Fix $s>2$ and $\Lambda>0$.  Let $b_k$ be determined before the fresh
Rademacher perturbation $\eta_k$ is drawn.  Conditional on the information
available at that time,
\begin{equation}
\Pp\left(
\norm{\widetilde r_k}_2^2<
\frac12(\norm{b_k}_2^2+\gamma^2M)
\right)\le2e^{-cM}.
\label{eq:perturbation-norm-event}
\end{equation}
If in addition $M^{-1}\sum_i|b_{k,i}|^s\le\Lambda^s$, then on the
complement of this event,
\begin{equation}
\frac{\norm{\widetilde r_k}_\infty}
     {\norm{\widetilde r_k}_2}
\le C_{\Lambda,\gamma}M^{1/s-1/2}.
\label{eq:perturbation-delocalization}
\end{equation}
Along $M/N\to\alpha$, if the empirical moment bound holds at every
step, a union bound gives \eqref{eq:perturbation-delocalization}
simultaneously over $O(N)$ steps with probability $1-e^{-\Omega(N)}$.
Moreover,
\begin{equation}
\max_{i\le M,\,\ell\le N}\left|\frac\gamma N
\sum_{k=1}^{\ell}\eta_{k,i}\right|
=O_{\Pp}\left(\gamma\sqrt{\frac{\log M}{N}}\right)=o_{\Pp}(1),
\label{eq:perturbation-terminal}
\end{equation}
and
\begin{equation}
\frac{\norm{b_k+\gamma\eta_k}_2}{\sqrt M}
\le\frac{\norm{b_k}_2}{\sqrt M}+\gamma.
\label{eq:perturbation-energy}
\end{equation}
\end{lemma}

\begin{proof}
Expanding the square,
\begin{equation}
\norm{b_k+\gamma\eta_k}_2^2
=\norm{b_k}_2^2+\gamma^2M+2\gamma\langle b_k,\eta_k\rangle.
\label{eq:perturbation-square}
\end{equation}
A Rademacher tail bound proves \eqref{eq:perturbation-norm-event}; the
exponent is of order $M$ because
$(\norm{b}_2^2+\gamma^2M)^2\ge4\gamma^2M\norm{b}_2^2$.  On the complement of
the event in \eqref{eq:perturbation-norm-event}, the denominator in
\eqref{eq:perturbation-delocalization} is at least
$c\gamma\sqrt M$, while $\norm{b_k}_\infty\le\Lambda M^{1/s}$.  Equation
\eqref{eq:perturbation-terminal} follows from Rademacher tail bounds and a union bound over the
coordinates and partial sums, and \eqref{eq:perturbation-energy} is the triangle inequality.
\end{proof}

Thus the perturbation increases the normalized Euclidean norm by at most
$\gamma$, gives the bound \eqref{eq:perturbation-delocalization}, and
contributes a vanishing terminal drift.

\subsection{Iteration and strong approximation}
\label{sec:iteration}

Fix an auxiliary exponent $s>12$, $0<\Lambda<\infty$, and
$0<\delta<\cstar$.
Suppose that every applied drift, chosen before observing the current
column, satisfies
\begin{align}
\frac{\norm{r_k}_2}{\sqrt N}&\le\cstar-\delta,
\label{eq:drift-cap}\\
\left(\frac1M\sum_i|r_{k,i}|^s\right)^{1/s}&\le\Lambda,
\label{eq:drift-Ls}\\
\frac{\norm{r_k}_\infty}{\norm{r_k}_2}
 &\le\Lambda M^{1/s-1/2},
\label{eq:drift-delocalization}
\end{align}
for $1\le k\le N$.  The last ratio is zero when $r_k=0$.
The middle expression is the empirical $L^s$ norm; the Euclidean norm
$\|r_k\|_2$ in the other two conditions is unnormalized.

We use the causal filtration that reveals the algorithm's randomness
as it is used.  For $1\le k\le N$, let $\cH_{k-1}$ be generated by the inputs,
the uniform random variables used to choose the signs, perturbations, and auxiliary coupling variables
through step $k-1$, together with the fresh perturbation $\eta_k$.
Let $\cH_N$ be generated by all these variables through step $N$.
The drift $r_k$ is $\cH_{k-1}$-measurable.  The current column and the
uniform random variable used to choose its sign are independent of $\cH_{k-1}$.
These sigma-fields do not contain the Gaussian variables added later
by strong approximation.

\begin{proposition}
\label{prop:iterated-coupling}
Assume $M\asymp N$ and $\mu\in\cM_4(K)$.  For all sufficiently large
$N$, on an enlarged probability space there are independent pairs
$(S_k,Z_k)$, with each $S_k$ a fair Rademacher sign, such that
\begin{equation}
\Law(S_k,Z_k\mid\cH_{k-1})
 =\Law(S_k)\otimes\nu_N^{\otimes M}.
\label{eq:product-innovations}
\end{equation}
For every $0\le\ell\le N$,
\begin{equation}
\frac1{\sqrt N}\sum_{k\le\ell}\eps_k\zeta_k
 =\frac1{\sqrt N}\sum_{k\le\ell}S_kZ_k
  +\frac1N\sum_{k\le\ell}r_k+E_\ell+Q_\ell,
\label{eq:iterated-decomposition}
\end{equation}
where $E_\ell$ is the accumulated predictable bias and $Q_\ell$ is
a martingale in $(\cH_\ell)$.  They obey
\begin{align}
\max_{\ell\le N}\norm{E_\ell}_\infty
 &\le C N^{-1/2+2/s},
\label{eq:iterated-bias}\\
\left\|\max_{\ell\le N}\norm{Q_\ell}_\infty\right\|_{L^{s}}
 &\le C_s N^{-(s-4)/(6s)}.
\label{eq:iterated-martingale}
\end{align}
The constants and the lower bound on $N$ depend only on $K,s,\Lambda,\delta$
and fixed positive lower and upper bounds on $M/N$.
\end{proposition}

\begin{proof}
Iterate the measurable kernel in \cref{lem:product-innovation} conditional
on the causal history.  The one-step identity is
\begin{equation}
\eps_k\zeta_{k,i}
 =S_kZ_{k,i}+\frac{r_{k,i}}{\sqrt N}+e_{k,i}+D_{k,i}.
\label{eq:conditional-one-step}
\end{equation}
By \eqref{eq:truncation-moments} and the drift bounds,
\begin{align}
|e_{k,i}|&\le C\left[
 \frac{|r_{k,i}|}{\sqrt N}
 (M^{1/s-1/2}+N^{-2/3}+N^{-1/2})+N^{-1}\right],
\label{eq:conditional-e-bound}\\
\E[D_{k,i}\mid\cH_{k-1}]&=0,
\label{eq:D-centered-conditional}\\
\E[D_{k,i}^2\mid\cH_{k-1}]&\le C N^{-1/2+1/s},\qquad
\E[|D_{k,i}|^{s}\mid\cH_{k-1}]
 \le C_{s} N^{(s-4)/3}.
\label{eq:D-conditional-moments}
\end{align}
The one-step residual $D_k$ is $\cH_k$-measurable and centered
conditionally on $\cH_{k-1}$.  The current perturbation has already
been revealed in $\cH_{k-1}$, whereas future algorithmic coins remain
independent of that field.
The conditional second-moment bound includes the cost of coupling
$\eps_k$ to $S_k$, as in \eqref{eq:D-moment-bounds}.
The fixed conditional law in \eqref{eq:product-innovations} makes the
pairs $(S_k,Z_k)$ independent, with the entire sequence $(S_k)$
independent of the array $(Z_{k,i})$.

Define $E_{\ell,i}=N^{-1/2}\sum_{k\le\ell}e_{k,i}$ and
$Q_{\ell,i}=N^{-1/2}\sum_{k\le\ell}D_{k,i}$.
Since $\norm{r_k}_\infty\le\Lambda M^{1/s}$,
\begin{equation}
\max_{\ell,i}|E_{\ell,i}|
 \le C\left(N^{-1/2+2/s}+N^{-2/3+1/s}
                 +N^{-1/2+1/s}+N^{-1/2}\right)
 \le C N^{-1/2+2/s}.
\label{eq:bias-calculation}
\end{equation}
For each coordinate, Burkholder--Rosenthal \cite{Burkholder73} gives
\begin{align}
\E\max_{\ell\le N}|Q_{\ell,i}|^{s}
 &\le C_{s}\E\left(
 \frac1N\sum_{k=1}^N\E[D_{k,i}^2\mid\cH_{k-1}]
 \right)^{s/2}
 +\frac{C_{s}}{N^{s/2}}\sum_{k=1}^N
 \E[|D_{k,i}|^{s}].
\label{eq:BR-coordinate}
\end{align}
Summing over $M=O(N)$ coordinates yields
\begin{equation}
\sum_i\E\max_{\ell\le N}|Q_{\ell,i}|^{s}
 \le C_{s}\left(N^{3/2-s/4}
       +N^{2-s/2+(s-4)/3}\right).
\label{eq:BR-summed}
\end{equation}
After taking the $s$th root, the two terms are
$N^{-1/4+3/(2s)}$ and $N^{-(s-4)/(6s)}$.  The first is smaller because
$s>12$.  This proves \eqref{eq:iterated-martingale}.  Only the jump term
uses the growing $s$th moment; the conditional variance bound uses a
uniform third-moment bound.
\end{proof}

We use Sakhanenko's strong approximation \cite{Sakhanenko85} in the
moment form stated in the author version of
\cite[Theorem~1, equation~(2)]{GotzeZaitsev12}; see also \cite{Shao95}.
For independent centered variables $V_k$ with finite $s$th moments, it
couples them to independent centered Gaussians $\Gamma_k$ of matching
variances so that
\begin{equation}
\E\max_{\ell\le N}\left|
 \sum_{k\le\ell}(V_k-\Gamma_k)\right|^{s}
 \le C_{s}\sum_{k\le N}\E|V_k|^{s}.
\label{eq:Sakhanenko}
\end{equation}
The bound applies to summands with different laws, with a constant
depending only on $s$; compare \cite{WaudbySmithLarssonRamdas25}.

Condition on the auxiliary signs and apply
\eqref{eq:Sakhanenko} independently in each row to
$S_k(Z_{k,i}-m_N)$.  These variables are conditionally independent and
centered, with variance $\sigma_N^2$.  The conditional Gaussian marginal
is the same array $(\sigma_N G_{k,i})$ for every sign sequence,
where all $G_{k,i}$ are independent standard Gaussians.  Write
$G_k=(G_{k,1},\ldots,G_{k,M})$.  For fixed $N$
there are finitely many sign sequences.  Choose a coupling for each row
and sign sequence.  After generating the pairs $(S_k,Z_k)$, sample the
Gaussian array from its regular conditional law, using independent
randomization across rows.  This preserves the original online trajectory.
The Gaussian marginal is the same product law for every sign sequence,
so the array is independent of $(S_k)$.  Its construction uses the
complete row through time $N$.
Summing the $s$th-moment bounds over rows gives an error
$C_{s} N^{-(s-4)/(6s)}$ after normalization by $\sqrt N$.
The mean path is deterministically bounded by
$|m_N|N^{-1/2}\max_\ell|\sum_{k\le\ell}S_k|
 \le\sqrt N|m_N|=O(N^{-1/2})$.
Replacing $\sigma_NG_{k,i}$ by $G_{k,i}$ costs
$O(N^{-2/3}\sqrt{\log N})$ in $L^s$ for the maximum over rows and
partial sums.  Consequently
\begin{equation}
\left\|\max_{i\le M,\,\ell\le N}
 \left|\frac1{\sqrt N}\sum_{k\le\ell}
       (S_kZ_{k,i}-G_{k,i})\right|\right\|_{L^{s}}
 \le C_{s} N^{-(s-4)/(6s)}.
\label{eq:all-row-strong-approximation}
\end{equation}

The normalized Gaussian partial sums
$N^{-1/2}\sum_{k\le\ell}G_{k,i}$ are the mesh values $B_i(\ell/N)$ of
independent Brownian motions.  At a deterministic split $k_0$, we perform this construction
separately before and after $k_0$, using independent auxiliary randomness.
Each pair $(S_k,Z_k)$ is independent of the causal history preceding step $k$,
so the Gaussian variables used after the split are independent of the
pre-split history and of the Gaussian variables attached before the split.
If $k_0/N<T$, we extend the pre-split Brownian paths to $T$ using increments
independent of the post-split construction.
The martingale estimates are proved in the causal filtration \emph{before}
conditioning on future auxiliary signs or adjoining the Gaussian coupling.
After this enlargement, the feedback and Euler comparisons are pathwise;
they use the Brownian marginal laws and the independence across the split.

The two coupling errors have $L^s$ norm $O(N^{-(s-4)/(6s)})$.
We use the fixed tolerance
\begin{equation}
\delta_N:=N^{-1/24}.
\label{eq:deltaN}
\end{equation}
For any prescribed $\theta>0$, taking $s$ sufficiently large makes the
probability that either error exceeds a fixed multiple of $\delta_N$
of order $O(N^{-\theta})$.  The predictable bias is $o(\delta_N)$.
The parameter choices and the remaining probability estimates are given
in \cref{app:parameters}; they use only the fourth moment of the original
entries.

\section{The online algorithm}
\label{sec:algorithm}

We construct the upper bound of \cref{thm:quantitative}.  Write
\[
D^{\mathrm{tr}}_{M,N}
:=\left\|\frac1{\sqrt N}\sum_{k=1}^N\eps_k\zeta_k\right\|_\infty
\]
for the discrepancy of the truncated signed sum.  The main step is the
following bound, which concerns a fixed margin above $R_\alpha$.

\begin{proposition}[Polynomial bounds for truncated inputs]
\label{prop:truncated-polynomial}
For every $\kappa_0>R_\alpha$, $K<\infty$, and $A>0$, parameters for
\cref{alg:online} can be chosen, independently of $\mu\in\cM_4(K)$, so that
\begin{equation}
\Pp(D^{\mathrm{tr}}_{M,N}>\kappa_0)
 +\E[(D^{\mathrm{tr}}_{M,N}-\kappa_0)_+]
 \le C_A N^{-A}.
\label{eq:truncated-polynomial}
\end{equation}
\end{proposition}

To deduce the quantitative theorem, choose
$\kappa_0\in(R_\alpha,\kappa)$ and set $g=\kappa-\kappa_0$.
Let $T_N$ be the maximum of the unsigned tail sums in
\eqref{eq:tail-envelope}.  The pathwise bound
$D_{M,N}\le D^{\mathrm{tr}}_{M,N}+T_N$ bounds the failure probability and
expected excess at $\kappa$ by the corresponding quantities for
$D^{\mathrm{tr}}_{M,N}$ above $\kappa_0$ and $T_N$ above $g$, added together.
The proposition bounds the first contribution by $C_A N^{-A}$, while
\eqref{eq:tail-fixed-gap} bounds the second by
$C_{g,K}\mathfrak t_\mu(N^{1/3})$.  The tail estimate and the full
deduction are proved in \cref{sec:tail-transfer}.

We now construct the algorithm in the proposition.  We first choose a
Brownian control satisfying a strict energy bound and a F\"ollmer drift for
the final block, then give the complete algorithm and its analysis.
\Cref{sec:truncated-assembly} completes the control choices.
For a desired failure exponent $\theta>0$, choose even integers $s,p$ with
\begin{equation}
s>\max\left\{12,\frac2\theta,8\theta+\frac{16}{3}\right\},
\qquad p>2s(1+\theta).
\label{eq:moment-choices}
\end{equation}
The exponent $s$ controls the truncated coupling and empirical drift
norms; $p$ supplies higher moments for the Brownian confinement step.
The cutoff tolerance is $\delta_N=N^{-1/24}$ as in \eqref{eq:deltaN}.
\Cref{app:parameters} verifies every resulting probability exponent.

\subsection{Choice of the Brownian control}
\label{sec:controls-slack}

\begin{proposition}
\label{prop:FJLNW-approximants}
Fix $c>0$, $1\le p<\infty$, and $T\in(0,1]$.  The optimal terminal
radius on horizon $T$ with pointwise $L^2$ drift bound $c$ is
\begin{equation}
r_{c,T}:=\sqrt T\,\mathcal R(c\sqrt T).
\label{eq:optimizer-scaled}
\end{equation}
There are bounded controls $u^{(n)}$ on $[0,T]$, progressively measurable
with respect to the natural Brownian filtration,
such that $\|u_t^{(n)}\|_{L^2}\le c$ for a.e.\ $t$, and, writing
$Y_T^{(n)}=B_T+\int_0^T u_t^{(n)}\,\dd t$,
\begin{equation}
\big\|(|Y_T^{(n)}|-r_{c,T})_+\big\|_{L^p}\longrightarrow0.
\label{eq:bounded-deficit}
\end{equation}
\end{proposition}

In \cref{app:controls}, we combine the bounded approximation result
of \cite[Lemma~4.15]{FiedlerJacksonLackerNilesWeed26} with its diffusion
representation.  This gives bounded controls adapted to a filtration in
which the driving process is Brownian.  Projection onto its natural
filtration preserves the pointwise $L^2$ bound and, by conditional Jensen,
can only decrease the moment of the distance to the target interval.

We use the simple predictable controls of \cite[Definition~4.2]{JoLee26},
with bounded Lipschitz functions of finitely many observations.  These
can be evaluated coordinate by coordinate from stored observations.

\begin{definition}
\label{def:finite-observation}
For a partition $0=t_0<t_1<\cdots<t_J=T$, we consider bounded simple
predictable controls of the form
\begin{equation}
\beta_t=F_j(B_{r_{j,1}},\ldots,B_{r_{j,m_j}}),
\qquad t\in(t_j,t_{j+1}].
\label{eq:finite-observation-form}
\end{equation}
Here $0\le j<J$, $m_j$ is the number of observations in block $j$,
$0\le r_{j,\ell}\le t_j$, and
$F_j:\R^{m_j}\to\R$ is bounded and globally Lipschitz.
Thus the drift on each block is determined by observations available at
its left endpoint.  A control of this form has a \emph{finite rational
description} if all times are rational and each $F_j$ is affine on the
cells of a finite rational polyhedral partition, with rational affine
coefficients.
\end{definition}

\begin{proposition}
\label{prop:finite-observation}
Let $T>0$ be rational and let $u$ be a bounded control
on $[0,T]$, progressively measurable with respect to the natural Brownian
filtration, satisfying
\[
\esssup_{0<t<T}\norm{u_t}_{L^2}\le c_0.
\]
Fix $c_1>c_0$, $1\le p<\infty$,
and $\eps>0$.  There is a bounded simple predictable control $\beta$
of the form \eqref{eq:finite-observation-form}, with a finite rational
description, such that
\begin{equation}
\esssup_{0<t<T}\norm{\beta_t}_{L^2}<c_1,
\qquad
\left\|\int_0^T(\beta_t-u_t)\,\dd t\right\|_{L^p}<\eps.
\label{eq:finite-observation-approximation}
\end{equation}
\end{proposition}

We prove the approximation for interval constraints in
\cref{app:controls}, preserving the uniform-in-time $L^2$ budget.
The corresponding result for nonnegative controls and a one-sided
terminal constraint is \cite[Proposition~4.3]{JoLee26}.

We next obtain the Brownian control used before the split.  For
$x\in\R$ and $a>0$, write
\begin{equation}
d_a(x):=\dist(x,[-a,a])=(|x|-a)_+.
\label{eq:interval-deficit}
\end{equation}
The corresponding projection is
$\Pi_{[-a,a]}(x):=\min\{a,\max\{-a,x\}\}$, applied coordinatewise to
vectors.  In particular, the magnitude of each projection correction
is its distance to the interval.

\begin{proposition}
\label{prop:brownian-preparation}
Fix $\alpha>0$, a finite $s\ge2$, and $a>R_\alpha$.  There are constants
$c_0<\ca$ and $h_0>0$ such that the following holds.  For every rational
$h\in(0,h_0)$, every
\begin{equation}
\frac{c_0}{\sqrt{1-h}}<c_\beta<\ca,
\label{eq:prep-budget-window}
\end{equation}
and every $\eps>0$, there is a bounded simple predictable control
$\beta$ on $[0,T]$, $T=1-h$, with a finite rational description and
\begin{equation}
\esssup_{0<t<T}\norm{\beta_t}_{L^2}<c_\beta,
\qquad
\norm{d_a(Y_T^\beta)}_{L^s}<\eps,
\label{eq:prep-output}
\end{equation}
where
\[
Y_t^\beta=B_t+\int_0^t\beta_r\,\dd r,\qquad 0\le t\le T.
\]
All displayed inequalities may be given fixed positive slack.
\end{proposition}

\begin{proof}
Continuity of $c\mapsto\mathcal R(c)$
(\cref{prop:brownian-facts}) and $\mathcal R(\ca)=R_\alpha<a$ give a
constant $c_0<\ca$ such that
\begin{equation}
\mathcal R(c_0)<a.
\label{eq:choose-c0-a}
\end{equation}
Choose $h_0>0$ so small that, for every $h<h_0$ and $T=1-h$,
\begin{equation}
\frac{c_0}{\sqrt T}<\ca,
\qquad
\sqrt T\,\mathcal R(c_0)<a.
\label{eq:time-compression-slack}
\end{equation}
Apply \cref{prop:FJLNW-approximants} on the rational horizon $T$ with
budget $c_0/\sqrt T$.  Its terminal radius is
\[
r_{c_0/\sqrt T,T}=\sqrt T\,\mathcal R(c_0)<a.
\]
It supplies a bounded admissible control $\widetilde u$ such that
\begin{equation}
\|d_a(Y_T^{\widetilde u})\|_{L^s}
\le\big\|(|Y_T^{\widetilde u}|-\sqrt T\,\mathcal R(c_0))_+\big\|_{L^s}
\label{eq:deficit-Wp-bound}
\end{equation}
is as small as needed.  Finally apply \cref{prop:finite-observation},
with approximation accuracy small compared with the gap in
\eqref{eq:prep-budget-window} and with $\eps$.

Since $d_a$ is one-Lipschitz, the resulting simple predictable control
satisfies \eqref{eq:prep-output}.
\end{proof}

\subsection{Confining the final block}
\label{sec:confinement}

The final block combines a correction toward the target interval with a
drift that confines the new noise in every coordinate.  The correction
drift is determined by the state at the split time and held fixed
throughout the final block.
We use F\"ollmer's construction \cite{Follmer85} in the form developed in
\cite[Lemmas~4.1--4.2]{FiedlerJacksonLackerNilesWeed26} and
\cite[Lemmas~C.1--C.3]{JoLee26}.  We stop the drift shortly before the
terminal time to control its sensitivity to the non-Gaussian coupling
error.

For $h>0$, let $\Gamma_h:=N(0,h)$ and write
\begin{equation}
P_tf(x):=\E[f(x+\sqrt t\,G)]
\label{eq:heat-semigroup}
\end{equation}
for the heat semigroup.

\begin{proposition}
\label{prop:compact-confinement}
Fix an even integer $p\ge4$ and $c_F>0$.  For every sufficiently small
$h>0$ there are an even log-concave density $f_h$ with respect to
$\Gamma_h$, supported on an interval $[-s_p(h,c_F),s_p(h,c_F)]$, and a
feedback
\begin{equation}
v_h(t,x):=\partial_x\log P_{h-t}f_h(x),
\qquad 0\le t<h,
\label{eq:Follmer-feedback}
\end{equation}
such that the strong solution of
\begin{equation}
\dd Z_t=v_h(t,Z_t)\,\dd t+\dd B_t,
\qquad Z_0=0,
\label{eq:Follmer-SDE}
\end{equation}
satisfies
\begin{align}
Z_h&\in[-s_p(h,c_F),s_p(h,c_F)]\quad\text{a.s.},
\label{eq:Follmer-support}\\
\sup_{0\le t<h}\norm{v_h(t,Z_t)}_{L^p}&\le \frac{c_F}{4},
\label{eq:Follmer-score-budget}\\
-\frac1{h-t}\le\partial_xv_h(t,x)&\le0,
\qquad 0\le t<h.
\label{eq:Follmer-derivative}
\end{align}
The factor $1/4$ leaves room for discretization: the Euler drifts below
have $L^p$ norm at most $c_F/2$ for all sufficiently large $N$.  Moreover,
\begin{equation}
s_p(h,c_F)\longrightarrow0\qquad(h\downarrow0).
\label{eq:confinement-radius-vanishes}
\end{equation}
\end{proposition}

\Cref{app:Follmer} constructs the density from an even, flat-topped
piecewise-linear cutoff and proves the proposition, including
localization at the zeros of $f_h$.

The derivative bound says that $v_h(t,\cdot)$ is nonincreasing; in
particular, it is dissipative.  If $\Delta=1/N$ and $h-t\ge\Delta$, the
Euler update $x\mapsto x+\Delta v_h(t,x)$ is nondecreasing and
one-Lipschitz.  We use the following recursion bound from
\cite[Lemma~C.4]{JoLee26}.

\begin{lemma}
\label{lem:dissipative-recursion}
Fix an integer $L\ge1$.  Let $e_0=\mathcal P_0=0$ and suppose,
for $0\le j<L$,
\begin{equation}
e_{j+1}=\vartheta_je_j+(\mathcal P_{j+1}-\mathcal P_j),
\qquad 0\le\vartheta_j\le1.
\label{eq:dissipative-recursion}
\end{equation}
Then
\begin{equation}
\max_{j\le L}|e_j|\le2\max_{j\le L}|\mathcal P_j|.
\label{eq:dissipative-bound}
\end{equation}
\end{lemma}

\begin{proof}
For $k\le j$, write
$\vartheta_{k:j}=\prod_{m=k}^{j-1}\vartheta_m$.  Summation by parts gives
\[
e_j=\mathcal P_j-\sum_{k=1}^{j-1}
(\vartheta_{k+1:j}-\vartheta_{k:j})\mathcal P_k.
\]
The coefficients in the sum are nonnegative and have total mass at most one.
\end{proof}

For a final block of $L_N$ steps, let $h_N=L_N/N\to h$.
For all sufficiently large $N$,
$\delta_N<h_N$ and $h_N$ lies in the range of
\cref{prop:compact-confinement}.  Let $j_N^*\in\{0,\ldots,L_N-1\}$
be the largest integer satisfying
\begin{equation}
j_N^*/N\le h_N-\delta_N.
\label{eq:cutoff-index}
\end{equation}
Define the truncated feedback
\begin{equation}
\bar v_N(t,x):=
\begin{cases}
v_{h_N}(t,x),&0\le t<j_N^*/N,\\
0,&j_N^*/N\le t\le h_N.
\end{cases}
\label{eq:truncated-feedback}
\end{equation}
The continuous truncated process $\bar Z$ is driven by the same Brownian
motion as the full process $Z$ and agrees with it until the cutoff.  After
the cutoff it uses zero drift.  Therefore
\begin{equation}
\bar Z_{h_N}-Z_{h_N}
=-\int_{j_N^*/N}^{h_N}v_{h_N}(t,Z_t)\,\dd t,
\label{eq:truncated-full-difference}
\end{equation}
and Minkowski's inequality gives
\begin{equation}
\norm{\bar Z_{h_N}-Z_{h_N}}_{L^p}
\le \frac{c_F}{4}(\delta_N+N^{-1}).
\label{eq:truncated-full-Lp}
\end{equation}
In particular, the truncation cost is linear in $\delta_N$; the Brownian
increment over the terminal interval cancels in this synchronous comparison.

Let $B_1,\ldots,B_M$ be independent Brownian motions, and let $\bar Z_i$
be the continuous truncated process driven by $B_i$.  On the same paths,
define the Euler scheme
\begin{equation}
\widehat Z_{0,i}=0,
\qquad
\widehat Z_{j+1,i}
=\widehat Z_{j,i}
+\frac1N\bar v_N(j/N,\widehat Z_{j,i})
+B_i((j+1)/N)-B_i(j/N).
\label{eq:truncated-Euler}
\end{equation}

The next lemma adapts \cite[Lemma~4.5 and Appendix~C.4]{JoLee26} to the
truncated feedback and records the polynomial probability bounds needed
here.

\begin{lemma}
\label{lem:Follmer-Euler}
With the construction of \cref{prop:compact-confinement},
\begin{equation}
\sup_{j<L_N}\norm{\bar v_N(j/N,\widehat Z_{j,i})}_{L^p}
\le \frac{c_F}{2}
\label{eq:Euler-request-Lp}
\end{equation}
for all sufficiently large $N$.  Coupling the Euler scheme with the
continuous truncated process through the same Brownian motion gives
\begin{equation}
\left\|\max_{j\le L_N}
|\widehat Z_{j,i}-\bar Z_i(j/N)|\right\|_{L^p}
\le\frac{C_pc_F}{N}.
\label{eq:Euler-one-row}
\end{equation}
For $M$ independent rows,
\begin{equation}
\left\|\max_{i\le M,\,j\le L_N}
|\widehat Z_{j,i}-\bar Z_i(j/N)|\right\|_{L^p}
\le C_pc_F\frac{M^{1/p}}{N}.
\label{eq:Euler-all-row}
\end{equation}
Before the cutoff, the corresponding drift difference satisfies
\begin{equation}
\left\|\max_{i\le M,\,j\le j_N^*}
|v_{h_N}(j/N,\widehat Z_{j,i})
 -v_{h_N}(j/N,\bar Z_i(j/N))|
\right\|_{L^p}
\le C_pc_F\frac{M^{1/p}}{N\delta_N}.
\label{eq:Euler-request-error}
\end{equation}
Finally, for every $\eps>0$,
\begin{equation}
\Pp\left(
\max_{i\le M}\dist(\widehat Z_{L_N,i},
[-s_p(h_N,c_F),s_p(h_N,c_F)])>\eps
\right)
\le C_{p,\eps}
\left(N^{1-p/24}+N^{1-p}\right).
\label{eq:Euler-terminal-tail}
\end{equation}
\end{lemma}

Along the full F\"ollmer diffusion, $v_{h_N}(t,Z_t)$ is an $L^p$-bounded
martingale; comparing its integral over a mesh interval with its left
endpoint produces an error martingale with integrand of order $1/N$, and
\cref{lem:dissipative-recursion} propagates the error stably.  The complete
calculation is given in \cref{app:Follmer}.

By \eqref{eq:moment-choices}, both terms on the right-hand side of
\eqref{eq:Euler-terminal-tail} are $O(N^{-\theta})$.  Moreover,
\begin{equation}
\frac{M^{1/p}}{N\delta_N}=O(N^{-23/24+1/p})=o(1),
\label{eq:Euler-window-check}
\end{equation}
and the probability that the drift difference in
\eqref{eq:Euler-request-error} exceeds a fixed constant is
$O(N^{-(23p/24-1)})=O(N^{-\theta})$.

\subsection{The algorithm and its computational cost}
\label{sec:online-algorithm}

Fix the control parameters: constants
\begin{equation}
0<a<\kappa,
\qquad
c_\beta,c_{\mathrm{corr}},c_F,\gamma>0,
\qquad
h\in(0,1),
\label{eq:transfer-constants}
\end{equation}
a rational split time $T=1-h$, with $h$ sufficiently small for
\cref{prop:compact-confinement}, and a bounded simple predictable
control $\beta$ on $[0,T]$ with a finite rational description, subject to
\begin{align}
\esssup_{0<t<T}\norm{\beta_t}_{L^2}&<c_\beta,\qquad c_\beta+\gamma<\ca,
\label{eq:transfer-control-budget}\\
c_{\mathrm{corr}}+c_F+\gamma&<\ca,
\label{eq:transfer-repair-budget}\\
\norm{d_a(Y_T^\beta)}_{L^s}&<h c_{\mathrm{corr}},
\label{eq:transfer-deficit}\\
a+s_p(h,c_F)&<\kappa,
\label{eq:transfer-radius}
\end{align}
with $\theta,s,p$ chosen as in \eqref{eq:moment-choices}.
Here $c_\beta$, $c_{\mathrm{corr}}$, and $c_F$ are the budgets for the
Brownian-control drift, the fixed correction, and the confinement drift,
respectively; $\gamma$ is the perturbation amplitude.  Let
\begin{equation}
k_0:=\lfloor TN\rfloor,
\qquad
L_N:=N-k_0,
\qquad
h_N:=\frac{L_N}{N}=h+O(N^{-1}).
\label{eq:k0-hN}
\end{equation}
Choose $\delta>0$ and a fixed compact interval $I_\alpha\subset(0,\infty)$
containing $\alpha$ in its interior so that
$\sqrt\lambda(c_\beta+\gamma)$ and
$\sqrt\lambda(c_{\mathrm{corr}}+c_F+\gamma)$ are both below
$\cstar-3\delta$ for every $\lambda\in I_\alpha$.  Choose a fixed bound $B_{\mathrm{req}}$ for the
empirical $L^s$ norms of the intended drifts on the comparison events
below, including a fixed allowance for comparison errors.  Choose
\[
\Lambda>\max\{B_{\mathrm{req}}+\gamma,
\sqrt2(B_{\mathrm{req}}+\gamma)/\gamma,1\}.
\]
The second term accounts for the lower norm bound in
\cref{lem:delocalization}; it ensures that the delocalization test below
holds on those comparison events.

Each intended drift $b_k$ is perturbed by $\gamma\eta_k$, where $\eta_k$
is a fresh Rademacher vector.  The perturbed drift is used only if it satisfies the
admissibility conditions
\begin{equation}
\frac{\norm{b_k+\gamma\eta_k}_2}{\sqrt N}
\le\cstar-\delta,
\qquad
\left(\frac1M\sum_i|b_{k,i}+\gamma\eta_{k,i}|^s\right)^{1/s}\le \Lambda,
\label{eq:adm-cap-Ls}
\end{equation}
\begin{equation}
\frac{\norm{b_k+\gamma\eta_k}_\infty}
     {\norm{b_k+\gamma\eta_k}_2}
\le \Lambda M^{1/s-1/2},
\label{eq:adm-deloc}
\end{equation}
where the ratio is defined to be zero when the drift vanishes.  If a
condition fails, the algorithm uses the auxiliary Rademacher drift
$\gamma\eta_k$ alone from that time onward.  When $M/N\in I_\alpha$,
this auxiliary drift satisfies
\eqref{eq:adm-cap-Ls}--\eqref{eq:adm-deloc} deterministically:
its normalized magnitude is $\gamma\sqrt{M/N}$, its empirical $L^s$
norm is $\gamma$, and its normalized largest coordinate is $M^{-1/2}$.
Thus \cref{prop:iterated-coupling} applies on every trajectory in this
range of dimensions.

The algorithm maintains the truncated signed state
$\widetilde X_k=N^{-1/2}\sum_{\ell\le k}\eps_\ell\zeta_\ell$,
with discrepancy $D^{\mathrm{tr}}_{M,N}=\norm{\widetilde X_N}_\infty$.
It also maintains a process $W$ approximating a Brownian motion before the
split, and a state $U$ for the fluctuation accumulated in the final block.
For the simple predictable control in \cref{def:finite-observation},
the drift before the split is
\begin{equation}
b_{k,i}=F_j\bigl(W_{\lfloor Nr_{j,1}\rfloor,i},\ldots,
                W_{\lfloor Nr_{j,m_j}\rfloor,i}\bigr),
\qquad t_j<k/N\le t_{j+1}.
\label{eq:causal-mesh-control}
\end{equation}
Every required observation is available before column $k$, since
$\lfloor Nr_{j,\ell}\rfloor\le\lfloor Nt_j\rfloor<k$.
Observations at mesh time $\ell$ are stored after updating $W_\ell$,
and $W_0=0$ is available initially.  Rounding block lengths costs
$O(N^{-1})$ because the control has finitely many bounded blocks.
The correction $r^{\mathrm{corr}}$ is determined by the split-time state
$\widetilde X_{k_0}$ and held fixed throughout the final block.

Fix a deterministic $N_0$, depending only on the fixed control data,
$I_\alpha$, and $K$, large enough for the conditions on the cutoff
and the variance after truncation.  If $N<N_0$ or $M/N\notin I_\alpha$, use independent
fair signs instead.  Thus the algorithm is well defined for every pair $(M,N)$,
and the updates below are used for all sufficiently large $N$ along any sequence
with $M/N\to\alpha$.

\begin{proposition}[Guarantee for the online algorithm]
\label{prop:upper}
Fix $\alpha>0$ and $\theta>0$, choose $s,p$ as in
\eqref{eq:moment-choices}, and let the data satisfy
\eqref{eq:transfer-constants}--\eqref{eq:transfer-radius}.
Then \cref{alg:online} satisfies
\begin{equation}
\Pp(D^{\mathrm{tr}}_{M,N}>\kappa)=O(N^{-\theta}),
\qquad
\E[(D^{\mathrm{tr}}_{M,N}-\kappa)_+]
\le C\left(N^{1/s}+\sqrt{\log N}\right)N^{-\theta/2}.
\label{eq:transfer-quantitative}
\end{equation}
The algorithm depends only on these data and a common fourth-moment
bound $K$.  The bounds are uniform over $\mu\in\cM_4(K)$ for all
sufficiently large $N$ and $M/N\in I_\alpha$, after shrinking $I_\alpha$
if necessary.
\end{proposition}

\begin{algorithmbox}{online vector-balancing algorithm\label{alg:online}}
\emph{Parameters:} the data
\eqref{eq:transfer-constants}--\eqref{eq:k0-hN}, the slack $\delta$, and the
admissibility constants $(\Lambda,s)$; use independent fair signs when
$N<N_0$ or $M/N\notin I_\alpha$.
\begin{enumerate}[label=\arabic*.,leftmargin=2em,itemsep=2pt,topsep=2pt]
\item Initialize $\widetilde X\gets0$, $W\gets0$, and
$\textsf{aux}\gets\textbf{false}$; store $W_0$.
\item For each $k=1,\ldots,N$, perform the following steps.
\begin{enumerate}[label=(\alph*),leftmargin=2em,itemsep=2pt,topsep=2pt]
\item If $k=k_0+1$ and $\textsf{aux}$ is false, set
$r^{\mathrm{corr}}\gets(\Pi_{[-a,a]}(\widetilde X)-\widetilde X)/h_N$
and $U\gets0$.
\par
If $\textsf{aux}$ is true, set $b_k\gets0$.
Otherwise compute the base drift from the past:
\[
b_k\gets
\begin{cases}
\text{the drift in \eqref{eq:causal-mesh-control}},&k\le k_0,\\
r^{\mathrm{corr}}+\bar v_N((k-k_0-1)/N,U),&k>k_0,
\end{cases}
\]
where the feedback is evaluated coordinatewise.
\item Draw a fresh vector $\eta_k\in\{\pm1\}^M$
with independent Rademacher coordinates.  If $\textsf{aux}$ is true or
$b_k+\gamma\eta_k$ fails \eqref{eq:adm-cap-Ls}--\eqref{eq:adm-deloc},
set $r_k\gets\gamma\eta_k$ and $\textsf{aux}\gets\textbf{true}$.
Otherwise set $r_k\gets b_k+\gamma\eta_k$.
\item Observe $\xi_k$, and set
$\zeta_{k,i}=\xi_{k,i}\1_{\{|\xi_{k,i}|\le N^{1/3}\}}$ for every $i$.
If $r_k=0$, choose an independent fair sign.  Otherwise put
$\rho=\norm{r_k}_2/\sqrt N$, $v=r_k/\norm{r_k}_2$, draw a fresh
$U_k'\sim\operatorname{Unif}[-1,1]$, and take $\eps_k=+1$ exactly when
$U_k'\le\psi_\rho(\langle v,\zeta_k\rangle)$.
Output $\eps_k$ irrevocably; this sign applies to the original column $\xi_k$.
\item Update $\widetilde X\gets\widetilde X+\eps_k\zeta_k/\sqrt N$.
For $k\le k_0$, also update
$W\gets W+\eps_k\zeta_k/\sqrt N-r_k/N$ and store every required
observation at mesh time $k$.
For $k>k_0$ with $\textsf{aux}$ still false, update
$U\gets U+\eps_k\zeta_k/\sqrt N-(r^{\mathrm{corr}}+\gamma\eta_k)/N$.
\end{enumerate}
\end{enumerate}
\end{algorithmbox}

All drifts are computed before the current vector is observed.  For fixed
theorem parameters and precomputed control data, each arrival uses
$O(M)$ arithmetic operations and scalar feedback evaluations, with $O(M)$
working memory including the fixed number of observations required by $\beta$.
We count evaluations of the smooth sign function and the logarithmic
derivative of the heat semigroup as scalar operations.  The latter can
be evaluated explicitly: the piecewise-linear cutoff used in
\cref{app:tent}, raised to the fixed even power $p$, is piecewise
polynomial, so its Gaussian convolution and derivative reduce to finitely
many truncated Gaussian moments.  These are expressible using arithmetic,
exponentials, and the Gaussian distribution function.  The finite rational
control $\beta$ is chosen once for the theorem parameters and is common
to the class with the prescribed moment bound.

\subsection{Analysis of the algorithm}
\label{sec:analysis}

We prove \cref{prop:upper}, working on the enlarged probability space that
carries the independent pairs $(S_k,Z_k)$ of \cref{prop:iterated-coupling} and the
Brownian strong approximation \eqref{eq:all-row-strong-approximation},
applied separately before and after the split.

\paragraph{Before the split.}
Let $\tau$ be the first failed test of an intended perturbed drift,
with $\tau=N+1$ if no test fails.  The comparisons below are stopped before this test; \cref{app:stopping}
proves that all tests pass on the basic comparison events.
The update of $W$ subtracts the applied drift, so
\eqref{eq:iterated-decomposition} gives for every $k\le k_0$, including
trajectories that use the auxiliary drift,
\begin{equation}
W_k=\frac1{\sqrt N}\sum_{\ell=1}^k S_\ell Z_\ell+E_k+Q_k.
\label{eq:W-product-decomposition}
\end{equation}
Couple the normalized partial sums of $S_kZ_k$ to independent Brownian
motions on the mesh using \eqref{eq:all-row-strong-approximation}.
Together with \eqref{eq:iterated-bias}--\eqref{eq:iterated-martingale},
this gives
\begin{equation}
\max_{i\le M,\,k\le k_0}|W_{k,i}-B_i(k/N)|
=O_{\Pp}\bigl(N^{-(s-4)/(6s)}\bigr).
\label{eq:W-B-comparison}
\end{equation}
For every fixed $\eps_0>0$, this maximum is at most
$\eps_0\delta_N$ with probability $1-O(N^{-\theta})$, by
\eqref{eq:coupling-tail-exponent}.

Let $Y_i^\beta$ be independent copies of the controlled Brownian path driven
by $B_i$.  The feedbacks in \cref{def:finite-observation} are globally
Lipschitz and depend on finitely many observation times.  The Brownian
modulus of continuity at the rounded observation times
(\cref{app:quantitative}), the comparison \eqref{eq:W-B-comparison}, and a
Riemann-sum estimate give
\begin{equation}
\1_{\{\tau>k_0\}}
\max_{i\le M}|\widetilde X_{k_0,i}-Y_{T,i}^\beta|
=o_{\Pp}(1).
\label{eq:split-time-comparison}
\end{equation}
Here the accumulated perturbation drift is $o_{\Pp}(1)$ by
\eqref{eq:perturbation-terminal}.  More explicitly, before the first
inadmissible time,
\[
\widetilde X_{k_0}=W_{k_0}+
\frac1N\sum_{k\le k_0}b_k
+\frac\gamma N\sum_{k\le k_0}\eta_k,
\]
and each of the three terms is compared with its Brownian counterpart.

The ideal drift before the split is bounded and has pointwise $L^2$ norm
strictly below $c_\beta$.  Concentration across independent coordinates,
over the finitely many feedback blocks, puts its empirical $L^2$ norm
below $c_\beta$ with fixed slack and bounds its empirical $L^s$ norm by
a fixed constant.
Equations \eqref{eq:W-B-comparison} and \eqref{eq:perturbation-energy},
together with \cref{lem:delocalization}, therefore show that all
admissibility conditions before the split hold with probability
$1-O(N^{-\theta})$.

Since $d_a$ is one-Lipschitz, the law of large numbers,
\eqref{eq:transfer-deficit}, and \eqref{eq:split-time-comparison} imply
\begin{equation}
\left(\frac1M\sum_{i=1}^M d_a(\widetilde X_{k_0,i})^s\right)^{1/s}
<h_N(c_{\mathrm{corr}}-\varepsilon_{\mathrm{corr}})
\label{eq:empirical-deficit}
\end{equation}
with probability $1-O(N^{-\theta})$, for some fixed $\varepsilon_{\mathrm{corr}}>0$; the ideal
variables have moments of every order because $\beta$ is bounded, and
Rosenthal's inequality with a sufficiently large moment exponent gives the
polynomial tail
(\cref{app:quantitative}).  By the definition of $r^{\mathrm{corr}}$ in
\cref{alg:online},
\begin{equation}
\left(\frac1M\sum_i|r_i^{\mathrm{corr}}|^s\right)^{1/s}
<c_{\mathrm{corr}}-\varepsilon_{\mathrm{corr}}.
\label{eq:correction-Ls}
\end{equation}

\paragraph{Final correction block.}
On $\{\tau>k_0\}$, condition on the history through time $k_0$ and work on
the event \eqref{eq:empirical-deficit}.  Let $\widehat Z_{j,i}$ be the independent
Euler approximations in \eqref{eq:truncated-Euler} to the F\"ollmer
diffusions with truncated drift, driven by the Gaussian increments
coupled to $S_kZ_k$ for $k>k_0$.  As
long as every drift has been admissible, subtracting
$r^{\mathrm{corr}}+\gamma\eta$ in the update of $U$ and using
\eqref{eq:iterated-decomposition} gives
\begin{equation}
U_{j+1}=U_j+
\frac1N\bar v_N(j/N,U_j)
+\frac{G_{k_0+j+1}}{\sqrt N}
+(\mathcal P_{j+1}-\mathcal P_j).
\label{eq:U-perturbed-Euler}
\end{equation}
Here $U_j$ is the final-block state after $j$ updates, with $U_0=0$,
and $\mathcal P_0=0$.  Starting at time $k_0$, the perturbation
$\mathcal P$ accumulates the predictable bias, the martingale error in the
coupling of $\eps_k\zeta_k$ and $S_kZ_k$, and the Gaussian
strong-approximation error.
Equations \eqref{eq:iterated-bias},
\eqref{eq:iterated-martingale}, and
\eqref{eq:all-row-strong-approximation} give
\begin{equation}
\max_{j\le L_N}\norm{\mathcal P_j}_\infty
=O_{\Pp}\bigl(N^{-(s-4)/(6s)}\bigr).
\label{eq:repair-perturbation}
\end{equation}

Before the cutoff the Euler map is nondecreasing and one-Lipschitz; after
the cutoff it is the identity.  Applying \cref{lem:dissipative-recursion}
coordinatewise yields
\begin{equation}
\max_{\substack{i\le M,\,j\le L_N\\k_0+j<\tau}}
|U_{j,i}-\widehat Z_{j,i}|
\le2\max_{j\le L_N}\norm{\mathcal P_j}_\infty.
\label{eq:U-Euler-comparison}
\end{equation}
Fix a constant $\eps_0>0$ smaller than the remaining drift budget in the final block.
On the event that the right-hand side of \eqref{eq:U-Euler-comparison} is at
most $\eps_0\delta_N/2$, \eqref{eq:Follmer-derivative} gives
\begin{equation}
\max_{\substack{i\le M,\,j\le j_N^*\\k_0+j<\tau}}
|\bar v_N(j/N,U_{j,i})-
  \bar v_N(j/N,\widehat Z_{j,i})|\le\eps_0.
\label{eq:actual-ideal-request-comparison}
\end{equation}
This event has probability $1-O(N^{-\theta})$ by
\eqref{eq:coupling-tail-exponent}.

The ideal drifts in \eqref{eq:Euler-request-Lp} are independent across
coordinates, conditionally on the history through time $k_0$.  Applying Rosenthal's
inequality to $|\bar v_N(j/N,\widehat Z_{j,i})|^s$ and a union bound over
the $O(N)$ final-block times (\cref{lem:empirical-moments}), with
probability $1-O(N^{-\theta})$ their empirical $L^s$ norms are bounded and
their empirical $L^2$ norms are at most $c_F/2+o(1)$.  Combining
\eqref{eq:correction-Ls}, \eqref{eq:actual-ideal-request-comparison}, and
Minkowski's inequality gives
\begin{equation}
\left(\frac1M\sum_i
|r_i^{\mathrm{corr}}+\bar v_N(j/N,U_{j,i})|^2\right)^{1/2}
\le c_{\mathrm{corr}}+c_F+\eps_0+o(1).
\label{eq:repair-empirical-L2}
\end{equation}
The strict budget \eqref{eq:transfer-repair-budget} and the perturbation
energy bound \eqref{eq:perturbation-energy} now show that the magnitude and
moment conditions in \eqref{eq:adm-cap-Ls} hold throughout the final block,
and delocalization holds by \cref{lem:delocalization}.  Thus the first
inadmissible time exceeds $N$ with probability $1-O(N^{-\theta})$.

On this event, summing the update of $U$ and using the definition of
$r^{\mathrm{corr}}$ gives the exact identity
\begin{equation}
\widetilde X_N=\Pi_{[-a,a]}(\widetilde X_{k_0})+U_{L_N}
+\frac\gamma N\sum_{k=k_0+1}^N\eta_k.
\label{eq:terminal-identity}
\end{equation}
By \eqref{eq:U-Euler-comparison} and \eqref{eq:Euler-terminal-tail},
\begin{equation}
\norm{U_{L_N}}_\infty
\le s_p(h_N,c_F)+o_{\Pp}(1).
\label{eq:U-terminal-bound}
\end{equation}
The perturbation term in \eqref{eq:terminal-identity} is $o_{\Pp}(1)$ by
\eqref{eq:perturbation-terminal}.  Since $a+s_p(h,c_F)<\kappa$ and the explicit radius
$s_p(\cdot,c_F)$ in \eqref{eq:tent-radius} is continuous near $h$,
$h_N\to h$ and equations
\eqref{eq:terminal-identity}--\eqref{eq:U-terminal-bound} give
$D^{\mathrm{tr}}_{M,N}\le\kappa$ with probability tending to one.

\paragraph{Exceptional events and integrability.}
Let $\cG_N$ be the intersection of the coupling, rounding,
concentration, perturbation, and confinement events, with errors inside
the fixed slack chosen above.  We then prove admissibility on this
event: the argument at the first failed test in \cref{app:stopping} proves
$\cG_N\subset\{\tau>N\}$ and then the terminal identity gives
$D^{\mathrm{tr}}_{M,N}\le\kappa$ on $\cG_N$.
The five polynomial failure exponents are
\begin{equation}
\frac{s}{8}-\frac23,
\qquad \frac{s}{8}-\frac23,
\qquad \frac p{24}-1,
\qquad \frac{p}{2s}-1,
\qquad \frac{23p}{24}-1,
\label{eq:failure-exponents}
\end{equation}
corresponding respectively to the coupling of $\eps_k\zeta_k$ and $S_kZ_k$,
the Gaussian approximation, switching off the F\"ollmer drift near the
terminal time, the empirical moments of that drift, and the Euler drift
comparison.  All exceed
$\theta$ by \eqref{eq:moment-choices}; the concentration
estimates for bounded pre-split drifts and the perturbation estimates
have exponential tails
(\cref{app:quantitative}).  Hence
\begin{equation}
\Pp(\cG_N^c)=O(N^{-\theta}),
\qquad D^{\mathrm{tr}}_{M,N}\le\kappa\quad\text{on }\cG_N.
\label{eq:good-event-tail}
\end{equation}

It remains to control the loss on $\cG_N^c$.  Because inadmissible drifts
are permanently replaced by the auxiliary Rademacher drift, every applied
drift, on every input, has a uniformly bounded empirical $L^s$ norm and
satisfies the magnitude and delocalization conditions.  Therefore the global
decomposition \eqref{eq:iterated-decomposition} and the strong
approximation \eqref{eq:all-row-strong-approximation} give
\begin{equation}
\norm{D^{\mathrm{tr}}_{M,N}}_{L^2}
\le C\left(N^{1/s}+\sqrt{\log N}\right);
\label{eq:global-L2-envelope}
\end{equation}
the drift contributes at most $CN^{1/s}$ coordinatewise, the maximum of
the normalized partial sums of $S_kZ_k$ is bounded by a Gaussian maximum plus the error from
strong approximation, and the two coupling
remainders are bounded by
\eqref{eq:iterated-bias}--\eqref{eq:iterated-martingale}
(\cref{app:quantitative}).  By Cauchy--Schwarz and
\eqref{eq:good-event-tail},
\begin{align}
\E[(D^{\mathrm{tr}}_{M,N}-\kappa)_+]
&\le\norm{D^{\mathrm{tr}}_{M,N}}_{L^2}\,\Pp(\cG_N^c)^{1/2}
\nonumber\\
&\le C\left(N^{1/s}+\sqrt{\log N}\right)N^{-\theta/2}
=o(1),
\label{eq:L1-excess-proof}
\end{align}
where the final equality is exactly $2/s<\theta$.  This proves
\cref{prop:upper}. \qed

\subsection{Choosing the control parameters}
\label{sec:truncated-assembly}

\begin{proof}[Proof of \cref{prop:truncated-polynomial}]
Set $\theta=2A+3$ and choose even integers $s,p$ satisfying
\eqref{eq:moment-choices}.
Fix $a\in(R_\alpha,\kappa_0)$ and apply
\cref{prop:brownian-preparation} to obtain $c_0<\ca$ and $h_0>0$.
Choose $\gamma>0$ so small that $c_0+4\gamma<\ca$, and choose $c_{\mathrm{corr}},c_F>0$
such that
\begin{equation}
c_{\mathrm{corr}}+c_F+2\gamma<\ca.
\label{eq:assembly-repair-budget}
\end{equation}
Choose a rational $h\in(0,h_0)$ sufficiently small that
\begin{equation}
\frac{c_0}{\sqrt{1-h}}<\ca-2\gamma,\qquad
a+s_p(h,c_F)<\kappa_0.
\label{eq:assembly-radius}
\end{equation}
Take $c_\beta\in(c_0/\sqrt{1-h},\ca-\gamma)$ and use
\cref{prop:brownian-preparation} with tolerance $hc_{\mathrm{corr}}$ for
the $L^s$ distance to $[-a,a]$ at time $T$.
The data satisfy
\eqref{eq:transfer-control-budget}--\eqref{eq:transfer-radius}, with
$\kappa_0$ in place of $\kappa$.  Apply \cref{prop:upper}.
The failure probability is $O(N^{-\theta})$, and
\eqref{eq:L1-excess-proof} bounds the expected excess by
$C(N^{1/s}+\sqrt{\log N})N^{-\theta/2}=O(N^{-A})$.
All strict budget inequalities persist for $M/N$ in a fixed sufficiently
small neighborhood of $\alpha$.
\end{proof}

\subsection{Effect of the discarded entries}
\label{sec:tail-transfer}

Apply the signs constructed for the truncated columns to the original
columns.  The following pathwise comparison holds for every choice of signs.
Set
\begin{equation}
W_{k,i}^{\mathrm{tail}}
 =|\xi_{k,i}|\1_{\{|\xi_{k,i}|>H_N\}},\qquad
T_N=\max_i\frac1{\sqrt N}\sum_{k=1}^N W_{k,i}^{\mathrm{tail}}.
\label{eq:tail-envelope}
\end{equation}
For every signing,
\begin{equation}
|D_{M,N}-D^{\mathrm{tr}}_{M,N}|\le T_N.
\label{eq:pathwise-tail-transfer}
\end{equation}
Let $\mathfrak t=\mathfrak t_\mu(H_N)$ and write
\[
R_N^{\mathrm{tail}}
 =\max_i\left|\frac1{\sqrt N}\sum_{k=1}^N
  (W_{k,i}^{\mathrm{tail}}-\E W_{k,i}^{\mathrm{tail}})\right|.
\]
Since $\E W_{k,i}^{\mathrm{tail}}\le K H_N^{-3}$,
\begin{equation}
T_N\le K N^{-1/2}+R_N^{\mathrm{tail}}.
\label{eq:tail-mean}
\end{equation}
Rosenthal's inequality applied to the independent tail
magnitudes, followed by $\max_i|x_i|^4\le\sum_i|x_i|^4$, gives
\begin{align}
\E(R_N^{\mathrm{tail}})^4
 &\le C\left[
 M\bigl(\E[(W_{1,1}^{\mathrm{tail}})^2]\bigr)^2
       +\frac MN\E[(W_{1,1}^{\mathrm{tail}})^4]\right]\nonumber\\
 &\le C\left[NH_N^{-4}\mathfrak t^2+\mathfrak t\right]
 =C\left[N^{-1/3}\mathfrak t^2+\mathfrak t\right]\le C_K\mathfrak t.
\label{eq:fourth-tail-Rosenthal}
\end{align}
For every fixed $g>0$, \eqref{eq:tail-mean} therefore implies, for all
sufficiently large $N$,
\begin{equation}
\Pp(T_N>g)+\E(T_N-g)_+
 \le C_g\E(R_N^{\mathrm{tail}})^4
 \le C_{g,K}\mathfrak t_\mu(N^{1/3}).
\label{eq:tail-fixed-gap}
\end{equation}
Indeed, $KN^{-1/2}\le g/2$, and Markov's inequality and
$(x-g/2)_+\le (2/g)^3x^4$ apply to $R_N^{\mathrm{tail}}$.

\begin{proof}[Proof of \cref{thm:quantitative}]
Fix $\kappa_0\in(R_\alpha,\kappa)$ and put $g=\kappa-\kappa_0$.
Use the algorithm for truncated inputs in \cref{prop:truncated-polynomial}.
The pathwise comparison gives
\begin{align*}
\Pp(D_{M,N}>\kappa)
 &\le\Pp(D^{\mathrm{tr}}_{M,N}>\kappa_0)+\Pp(T_N>g),\\
(D_{M,N}-\kappa)_+
 &\le(D^{\mathrm{tr}}_{M,N}-\kappa_0)_++(T_N-g)_+.
\end{align*}
Equations \eqref{eq:truncated-polynomial} and \eqref{eq:tail-fixed-gap}
prove \eqref{eq:main-tail-modulus}.  For a fixed law, the fourth-moment
tail tends to zero.  Hence
\begin{equation}
\E D_{M,N}\le\kappa+\E(D_{M,N}-\kappa)_+=\kappa+o(1).
\label{eq:expected-from-L1}
\end{equation}
Letting $\kappa\downarrow R_\alpha$ gives the required upper bound on
the limiting optimal expectation.  Taking suprema of the tail modulus
proves the stated uniformity over classes with uniformly integrable
fourth moments.

All parameters depend on the law only through $K$.
Truncation itself takes $O(M)$ arithmetic operations; the remaining
complexity was verified after \cref{alg:online}.

For each fixed law the seed can also be fixed.  Average the nonnegative
loss $\1_{\{D_{M,N}>\kappa\}}+(D_{M,N}-\kappa)_+$ over the external seed.
Some fixed seed has conditional expected loss no larger than this
average.  Choosing such a seed for each $N$ preserves both conclusions.
The randomized algorithm is common to the whole class; a deterministic
seed is chosen separately for each law.
\end{proof}

\section{Lower bounds}
\label{sec:lower}

The lower bounds assume mean zero, unit variance, and a finite fourth
moment.  We recall the compactness argument and deduce infeasibility below
$R_\alpha$.

\subsection{Empirical path measures}
\label{sec:sampled-row}

With respect to the filtration
$(\sigma(\mathsf U,\xi_1,\ldots,\xi_k))_{k=0}^N$, write the Doob
decomposition in each coordinate as
\begin{equation}
X_{k,i}=A_{N,i}(k/N)+B_{N,i}(k/N),
\qquad 0\le k\le N,
\label{eq:selected-row-Doob}
\end{equation}
where $A_{N,i}$ is the predictable drift and $B_{N,i}$ the centered
martingale part, both linearly interpolated between mesh points.  On
$\mathcal C^2=C([0,1],\R)^2$, define the random empirical measure
\[
\widehat\nu_N:=\frac1M\sum_{i=1}^M\delta_{(A_{N,i},B_{N,i})}.
\]
Here $\delta_x$ denotes unit mass at $x$.  We write $\mathcal P(E)$ for
the Borel probability measures on a metric space $E$, equipped with weak
convergence.  The measure $\widehat\nu_N$ is random; its law is therefore
an element of $\mathcal P(\mathcal P(\mathcal C^2))$.  We retain this
outer probability law to bound the probability of simultaneous feasibility.

We use the following rectangular adaptation of
\cite[Sections~3.2--3.4]{FiedlerJacksonLackerNilesWeed26}, especially
Proposition~3.5 and Theorem~3.7, with the proportional scaling discussed
after their equations~(1.3)--(1.5).

\begin{proposition}
\label{prop:rectangular-compactness}
Assume that the coordinate law has mean zero, unit variance, and a finite
fourth moment.  For every sequence of randomized online algorithms and every
sequence $M/N\to\alpha$, the laws of $\widehat\nu_N$ are tight in
$\mathcal P(\mathcal P(\mathcal C^2))$, with weak convergence at both levels.
Every limit $Q$ is supported on measures $m$ under which the canonical
processes $(A,B)$ satisfy
\begin{equation}
A_t=\int_0^t a_s\,\dd s,\qquad 0\le t\le1,
\label{eq:weak-Brownian-limit}
\end{equation}
where $B$ is a Brownian motion in the canonical joint filtration, $a$ is
progressively measurable, and
\begin{equation}
\E_m a_t^2\le\frac{2}{\pi\alpha}=\ca^2
\qquad\text{for a.e.\ }t\in[0,1].
\label{eq:weak-energy-cap}
\end{equation}
In addition,
\begin{equation}
\liminf_{M/N\to\alpha}\inf_{\eps\,\mathrm{online}}\E_\mu D_{M,N}\ge R_\alpha.
\label{eq:FJLNW-expected-lower}
\end{equation}
\end{proposition}

\begin{proof}
The cited Theorem~3.7 is stated for the square model on $[0,1]$.
Apply its tightness and martingale problem argument with dimension $M$,
$N$ steps of duration $1/M$, and entry normalization $M^{-1/2}$.
The resulting horizon is $N/M\to1/\alpha$.  The estimates of their
Sections~3.2--3.3 hold on any fixed bounded horizon; their drift truncation
argument gives the same pointwise budget $2/\pi$.  The time change
$s=tN/M$ and spatial multiplication by $\sqrt{M/N}$ put the paths on
$[0,1]$ in our normalization.  Brownian variance becomes one, and the
drift budget becomes $(N/M)(2/\pi)\to2/(\pi\alpha)$.
Thus every limiting distribution is supported on the stated admissible
laws.  An independent algorithmic seed can be included in the initial
filtration: the conditional moment and martingale estimates used in that
argument are unchanged.
The expected lower bound follows by their Section~3.4 with this scaling;
it also follows from \cref{prop:hp-lower} below and nonnegativity of $D_{M,N}$.
\end{proof}

\subsection{Infeasibility below the Brownian value}
\label{sec:hp-lower}

Projection onto the natural Brownian filtration
(\cref{lem:Brownian-projection}) preserves the closed interval constraint
and the bound on the drift's second moment.

\begin{proposition}
\label{prop:hp-lower}
Suppose $\mu$ has mean zero, unit variance, and a finite fourth moment.
For every $0\le\kappa<R_\alpha$ and every sequence $M/N\to\alpha$,
\begin{equation}
\sup_{\eps\,\mathrm{online}}\Pp(D_{M,N}\le\kappa)\longrightarrow0.
\label{eq:hp-success-assumption}
\end{equation}
\end{proposition}

\begin{proof}
For $m\in\mathcal P(\mathcal C^2)$ let
\[
\mathcal E(m):=\|A(1)+B(1)\|_{L^\infty(m)}.
\]
The set $\{m:\mathcal E(m)\le\kappa\}$ is closed: it consists of measures
giving mass one to the closed path set
$\{(a,b):|a(1)+b(1)|\le\kappa\}$.  In particular, $\mathcal E$ is lower
semicontinuous and $\mathcal E(\widehat\nu_N)=D_{M,N}$.
For an arbitrary sequence of online algorithms, take a subsequential
limit $Q$ supplied by \cref{prop:rectangular-compactness}.  For $Q$-almost
every $m$, \cref{lem:Brownian-projection} shows
$\mathcal E(m)\ge\mathcal R(\ca)=R_\alpha$: any smaller finite terminal
interval would also be attainable in the natural Brownian filtration.
Portmanteau therefore gives
\[
\limsup_N\Pp(D_{M,N}\le\kappa)
\le Q\{m:\mathcal E(m)\le\kappa\}=0.
\]
Every subsequence has such a further subsequence, so the whole sequence
converges.  If the supremum over algorithms failed to converge to zero,
choosing an almost maximizing algorithm at each $N$ would contradict the
statement just proved.
\end{proof}

\begin{corollary}
\label{cor:expected-to-probability}
Under the assumptions of \cref{prop:hp-lower}, any sequence of online
algorithms with $\E D_{M,N}\to R_\alpha$ satisfies
$\E|D_{M,N}-R_\alpha|\to0$, and hence $D_{M,N}\to R_\alpha$ in probability.
\end{corollary}

\begin{proof}
For $0<\eps<R_\alpha$,
$\E(R_\alpha-D_{M,N})_+\le\eps+
R_\alpha\Pp(D_{M,N}<R_\alpha-\eps)$ has limsup at most $\eps$.
Send $\eps\downarrow0$ and use
\[
\E(D_{M,N}-R_\alpha)_+
=\E D_{M,N}-R_\alpha+\E(R_\alpha-D_{M,N})_+\longrightarrow0.
\]
\end{proof}

For Gaussian inputs, \cref{cor:expected-to-probability} applies to the online algorithms
in \cite{FiedlerJacksonLackerNilesWeed26} whose expected discrepancy converges
to $R_\alpha$.
Their Gaussian upper bound therefore also gives feasibility with
probability tending to one for every $\kappa>R_\alpha$.

\begin{proof}[Proof of \cref{thm:main}]
For part (i), apply \cref{thm:quantitative} with any fixed $A>0$, for
example $A=1$.  The fourth-moment tail tends to zero for each fixed law.
The expected upper bound follows from \eqref{eq:expected-from-L1} by
first taking a limsup and then letting $\kappa\downarrow R_\alpha$.
The matching lower bound is \eqref{eq:FJLNW-expected-lower}, and
\cref{prop:hp-lower} proves part (ii).
\end{proof}

\section{Moment bounds and independent entries}
\label{sec:extensions}

\subsection{Bounds under higher moments}

\begin{proof}[Proof of \cref{cor:higher-moment-rate}]
Let $q>4$ be as in the corollary.  Under
$\E|\xi|^q\le K$, Rosenthal with exponent $q$ improves
\eqref{eq:fourth-tail-Rosenthal} to
\begin{align}
\E(R_N^{\mathrm{tail}})^q
 &\le C_q\left[
 M\bigl(\E[(W_{1,1}^{\mathrm{tail}})^2]\bigr)^{q/2}
 +MN^{1-q/2}\E[(W_{1,1}^{\mathrm{tail}})^q]\right]\nonumber\\
 &\le C_{q,K}\left[
 N^{1-q(q-2)/6}+N^{2-q/2}\right]
 \le C_{q,K}N^{-(q-4)/2}.
\label{eq:high-moment-tail-Rosenthal}
\end{align}
Here $\E[(W^{\mathrm{tail}})^2]\le K H_N^{2-q}$ and
$\E[(W^{\mathrm{tail}})^q]\le K$; the first exponent is no larger than the
second because $(q-2)(q-3)>0$.
For a fixed gap, \eqref{eq:tail-fixed-gap} with $q$th moments bounds both
the failure probability and expected excess by this rate.
Choose the accuracy in \cref{prop:truncated-polynomial} to exceed
$(q-4)/2$ and use the same split
$\kappa_0\in(R_\alpha,\kappa)$.  The previous proof gives
\eqref{eq:higher-moment-rate}.
\end{proof}

\subsection{Independent entries}
\label{sec:independent-array}

\begin{theorem}
\label{thm:independent-array}
Let $M_N/N\to\alpha$, and let the entries $\xi_{N,k,i}$ be mutually
independent, each with mean zero and variance one.  Their distributions
may differ.
Assume their fourth moments are uniformly integrable:
\begin{equation}
\mathfrak t_4(H):=\sup_{N,k,i}
 \E[|\xi_{N,k,i}|^4\1_{\{|\xi_{N,k,i}|>H\}}]\longrightarrow0
 \quad(H\to\infty).
\label{eq:array-tail-modulus}
\end{equation}
Then the optimal expected online discrepancy converges to $R_\alpha$.
For every $\kappa>R_\alpha$ and $A>0$, the same construction, depending
only on $(\alpha,\kappa,A)$ and a common fourth-moment bound, satisfies
\eqref{eq:main-tail-modulus} with $\mathfrak t_4$ in place of
$\mathfrak t_\mu$.  A common $q$th-moment bound, $q>4$, gives the rate in
\cref{cor:higher-moment-rate}.  Under a common fourth-moment bound alone,
the supremal success probability tends to zero for every $\kappa<R_\alpha$.
\end{theorem}

\begin{proof}
Write $K=\sup_{N,k,i}\E|\xi_{N,k,i}|^4<\infty$ and suppress the index $N$.

\paragraph{Upper bound.}
Use the truncation $\xi_{k,i}\1_{\{|\xi_{k,i}|\le H_N\}}$ with $H_N=N^{1/3}$.
The coordinate means $m_{k,i}$ and variances $\sigma_{k,i}^2$ satisfy
\eqref{eq:truncation-moments} uniformly in $k,i$.
For the independent truncated coordinates, the
proof of \cref{lem:low-influence-response} gives
\[
|e_i|\le C\rho\left[
 |v_i|(\norm v_\infty+\Delta+|d|)+|m_{k,i}|\right],\qquad
|t|\le C\rho(\norm v_\infty+|d|),
\]
where $\Delta=\max_i|\sigma_{k,i}^2-1|=O(N^{-2/3})$ and
$d=\sum_iv_im_{k,i}=O(N^{-1/2})$.
Indeed, the centered projection variance is
$\sum_i v_i^2\sigma_{k,i}^2=1+O(\Delta)$, and the
Berry--Esseen bound for independent summands is unchanged.
The conditional densities and coupling kernels use the individual laws
$\nu_{k,i}$; their second marginal is $\bigotimes_i\nu_{k,i}$ for every
history.  Thus the estimates of \cref{prop:iterated-coupling} remain valid:
at time $k$, conditionally on the past, $S_k$ is a fair Rademacher sign
independent of $Z_k$, whose law is $\bigotimes_i\nu_{k,i}$.

Conditionally on the auxiliary signs, apply
\eqref{eq:Sakhanenko} to
$S_k(Z_{k,i}-m_{k,i})$, with matching Gaussian
$\sigma_{k,i}G_{k,i}$.  The mean correction has maximum at most
$\sqrt N\max_{k,i}|m_{k,i}|=O(N^{-1/2})$.
The variance correction has $L^{s}$ norm of the maximum over rows and partial sums
$O(N^{-2/3}\sqrt{\log N})$: in each row it is a Gaussian martingale
with total variance $O(N^{-4/3})$.  Hence independent standard Brownian motions
drive the comparison, and
\cref{prop:truncated-polynomial} holds uniformly for these arrays.

For the unsigned sum in \eqref{eq:tail-envelope}, center each entry
separately in the definition of $R_N^{\mathrm{tail}}$.  The maximum
normalized row mean is at most $KN^{-1/2}$, and Rosenthal's inequality
gives
\[
\E(R_N^{\mathrm{tail}})^4
\le C\left[N^{-1/3}\mathfrak t_4(H_N)^2+\mathfrak t_4(H_N)\right]
\le C_K\mathfrak t_4(H_N).
\]
At a fixed positive gap, this bounds the failure probability and expected
excess, and gives the upper bound on expected discrepancy.
Under a common $q$th moment, the
calculation in \eqref{eq:high-moment-tail-Rosenthal} is also unchanged.

\paragraph{Lower bound.}
Assume a common fourth-moment bound.  We replace the use of permutation
symmetry in \cite[Proposition~3.5]{FiedlerJacksonLackerNilesWeed26}
by an estimate for independent coordinates.  For an arbitrary online algorithm, let
$\cF_k=\sigma(\mathsf U,\xi_1,\ldots,\xi_k)$, where the seed
$\mathsf U$ is independent of all entries.  Condition on the past and set
\[
u_k=\E[\eps_k\xi_k\mid\cF_{k-1}],\qquad
d_k=\eps_k\xi_k-u_k,\qquad a_{k,i}=\sqrt N\,u_{k,i}.
\]
Orthonormality of the centered unit-variance coordinates in $L^2$ gives
\begin{equation}
\norm{u_k}_2\le1,\qquad
\E[d_kd_k^\top\mid\cF_{k-1}]=I-u_ku_k^\top,\qquad
\E[|d_{k,i}|^4\mid\cF_{k-1}]\le16K.
\label{eq:array-martingale-bounds}
\end{equation}
Here $I$ is the $M\times M$ identity matrix; in particular, the conditional
covariance is bounded above by $I$.  Define $A_i,B_i$ by the linearly interpolated
Doob decomposition \eqref{eq:selected-row-Doob}, suppressing the index $N$;
then $\dot A_i(t)=a_{k,i}$ on $((k-1)/N,k/N)$.
Thus the drift energy averaged across rows is at most $N/M$.  This bound
and the martingale fourth-moment estimates give tightness of the random
empirical $(A,B)$ path measures, as in \cref{prop:rectangular-compactness}.
For the linearly interpolated martingales,
$\E|B_i(t)-B_i(s)|^4\le C_K|t-s|^2$ uniformly in $i$.

To obtain the sharp energy bound, take $L=\log M$ and truncate the drift:
$b_{k,i}=a_{k,i}\1_{\{|a_{k,i}|\le L\}}$.
If $\|b_k\|_2\ge M^{1/4}$, put $v=b_k/\|b_k\|_2$.  Then
$\|v\|_2=1$ and $\|v\|_\infty\le L/M^{1/4}$, so
\begin{align}
\frac{\|b_k\|_2}{\sqrt N}
 &=\E[\eps_k\langle v,\xi_k\rangle\mid\cF_{k-1}]
 \le\E[|\langle v,\xi_k\rangle|\mid\cF_{k-1}]
 \le\cstar+C_K L/M^{1/4}.
\label{eq:array-sharp-energy}
\end{align}
The last step is the Wasserstein Berry--Esseen inequality for independent
summands with different distributions.  If $\|b_k\|_2<M^{1/4}$, its
empirical squared norm is less than $M^{-1/2}$.  Consequently,
$M^{-1}\sum_i b_{k,i}^2\le\ca^2+o(1)$ at every time, deterministically
and uniformly over online algorithms.
Let $a_i(t)=a_{k,i}$ and $b_i(t)=b_{k,i}$ for
$(k-1)/N<t\le k/N$, and set
$\widetilde A_i(t)=\int_0^t b_i(r)\,\dd r$.  Then
\[
\frac1M\sum_i\norm{A_i-\widetilde A_i}_\infty
 \le\frac1{ML}\sum_i\int_0^1 a_i(t)^2\,\dd t
 \le\frac{N}{ML}\longrightarrow0.
\]
Thus truncation preserves every empirical limit.  Lower semicontinuity
of the weighted energies $\int_0^1 w(t)|\dot A_t|^2\,\dd t$,
first for an arbitrary nonnegative continuous weight $w$ and then for
a countable dense family of such weights, shows that almost every limiting empirical
law has $\E a_t^2\le\ca^2$ for almost every $t$.

To identify the Brownian motion in each limiting law, use bounded
continuous functions of finitely many past evaluations of $(A_i,B_i)$,
together with a bounded $C^3$ function $f:\R\to\R$ whose first three
derivatives are bounded, in the averaged discrete martingale problem.
By \eqref{eq:array-martingale-bounds}, the averaged linear martingale
term has variance $O(1/M)$.  The total error in its quadratic
compensator is at most
$C(MN)^{-1}\sum_k\norm{u_k}_2^2\le C/M$, where $C$ depends on the test.
The centered quadratic term uses the identity
\[
d_{k,i}^2-\E[d_{k,i}^2\mid\cF_{k-1}]
 =(\xi_{k,i}^2-1)-2u_{k,i}d_{k,i}.
\]
For predictable weights $|g_i|\le C$, independence of the current raw
coordinates and \eqref{eq:array-martingale-bounds} give conditional
variance at most $C_K M$ for the weighted sum of these terms.
After the factor $1/(MN)$ and summation over $N$ steps its variance is
$O(1/(MN))$.  The averaged cubic Taylor remainder has expectation
$O_K(N^{-1/2})$.
Round the observation times in the cylinder functions down to the mesh.
Tightness makes the rounding and Riemann-sum errors vanish.  The resulting bounded
continuous empirical test functionals converge, and the residual
vanishes in probability for a countable determining family of tests.
Every random empirical limit is supported
on laws under which $B$ is Brownian in the \emph{joint} $(A,B)$
filtration.  Combined with the sharp bound on drift energy, these are exactly the
admissible laws of \cref{prop:rectangular-compactness}.  Projection onto the Brownian
filtration and closedness of the terminal support condition then prove
infeasibility below $R_\alpha$, and hence the expected lower bound.
\end{proof}

\section{The online threshold of the symmetric binary perceptron}
\label{sec:sbp}

\subsection{Monotonicity of the discrepancy curve}
\label{sec:monotonicity}

By \cite[Proposition~2.8]{FiedlerJacksonLackerNilesWeed26} and the
scaling identity \eqref{eq:Ralpha-cap}, the function $\mathcal R$ is
convex and nonincreasing on $(0,\infty)$.  The small-time asymptotic for
$V_T^\infty$ gives $\mathcal R(c)\to\infty$ as $c\downarrow0$, whereas
boundedness of $V_T^\infty$ gives $\mathcal R(c)\to0$ as $c\to\infty$.

\begin{lemma}
\label{lem:strict-monotonicity}
The function $\mathcal R$ is continuous and strictly decreasing on
$(0,\infty)$.  Consequently, $\alpha\mapsto R_\alpha$ is continuous and
strictly increasing from zero to infinity.
\end{lemma}

\begin{proof}
A finite convex function on an open interval is continuous.  If a convex
nonincreasing function is constant on a nontrivial interval, its one-sided
slopes are zero there.  Convexity makes later slopes nonnegative, while
monotonicity makes them nonpositive, so the function remains constant
thereafter.  For $\mathcal R$ this would force the constant to be zero,
since $\mathcal R(c)\to0$ as $c\to\infty$.  But the entropy lower bound of
\cite[Proposition~6.5]{FiedlerJacksonLackerNilesWeed26}, after scaling,
gives $\mathcal R(c)>0$ for every finite $c>0$.
Thus $\mathcal R$ is strictly decreasing.  The identity
$R_\alpha=\mathcal R(\cstar/\sqrt\alpha)$ proves the claim for $R_\alpha$.
\end{proof}

\subsection{The threshold and its asymptotics}
\label{sec:threshold}

For a law $\mu$, define
\begin{equation}
\aonmu(\kappa)
:=\sup\left\{\alpha>0:
\begin{array}{l}
\text{for every }M/N\to\alpha\text{ there are online algorithms with}\\[-1mm]
\Pp_\mu(D_{M,N}\le\kappa)\to1
\end{array}\right\}.
\label{eq:SBP-threshold-definition}
\end{equation}

\begin{proof}[Proof of \cref{cor:sbp}]
By \cref{thm:main}, a fixed density $\alpha$ is achievable at margin
$\kappa$ whenever $R_\alpha<\kappa$, and is infeasible when
$R_\alpha>\kappa$.  By \cref{lem:strict-monotonicity}, the equation $R_\alpha=\kappa$ has
a unique solution.  The definition as a supremum identifies that
solution with $\aonmu(\kappa)$ for every admissible $\mu$.
\end{proof}

We invert the asymptotics at the two ends of the curve.  As
$\alpha\downarrow0$,
\begin{equation}
R_\alpha=\sqrt\alpha\,V_{1/\alpha}^{\infty}
\sim \ell_\star\sqrt\alpha,
\label{eq:small-alpha-proof}
\end{equation}
while as $\alpha\to\infty$, the asymptotic from
\cite[Propositions~6.3--6.5]{FiedlerJacksonLackerNilesWeed26},
\begin{equation}
V_T^{\infty}\sim\sqrt{2T\log(1/T)}
\qquad(T\downarrow0),
\label{eq:FJLNW-small-time}
\end{equation}
gives
\begin{equation}
R_\alpha\sim\sqrt{2\log\alpha}.
\label{eq:large-alpha-proof}
\end{equation}
Inverting these equivalents using \cref{lem:strict-monotonicity} yields
\begin{equation}
\aonmu(\kappa)\sim \ell_\star^{-2}\kappa^2
\quad(\kappa\downarrow0),
\qquad
\log\aonmu(\kappa)\sim\frac{\kappa^2}{2}
\quad(\kappa\to\infty),
\label{eq:sbp-asymptotics}
\end{equation}
for every $\mu$ satisfying the assumptions of \cref{thm:main}.
In particular, $\ell_\star^{-2}$ is the exact leading constant in the
small-margin online threshold.  The curve passes through
$\aonmu(\kappa)=1$ precisely at $\kappa=V_1^\infty$.
The explicit bounds of \cite[Section~6]{FiedlerJacksonLackerNilesWeed26}
give
\[
1.09\le V_1^\infty\le1.97.
\]

\paragraph{Comparison with the offline capacity.}
For Gaussian and Rademacher entries, the offline capacity is
\[
\alpha_{\mathrm{off}}(\kappa)
=-\frac{\log 2}{\log p_\kappa},
\qquad
p_\kappa=\Pp(|G|\le\kappa),\quad G\sim N(0,1).
\]
Aubin, Perkins, and Zdeborov\'a \cite{AubinPerkinsZdeborova19} identify
this formula under an analytic condition.  Perkins and Xu
\cite[Theorem~2]{PerkinsXu21} establish the Gaussian sharp threshold
under that condition; Abbe, Li, and Sly
\cite[Theorem~2.3 and Lemma~3.4]{AbbeLiSly21} establish the needed analytic
inequality and prove the binary sharp threshold.  The Gaussian formula
also follows from that binary threshold and the universality theorem of
\cite[Theorem~1.11]{NakajimaSun23}.
Since $p_\kappa=\sqrt{2/\pi}\,\kappa(1+O(\kappa^2))$,
\[
\alpha_{\mathrm{off}}(\kappa)\sim\frac{\log2}{\log(1/\kappa)},
\qquad
\frac{\aonmu(\kappa)}{\alpha_{\mathrm{off}}(\kappa)}
\sim\frac{\kappa^2\log(1/\kappa)}{\ell_\star^2\log2}
\longrightarrow0
\quad(\kappa\downarrow0).
\]

\paragraph{Open questions.}
What is the weakest tail condition under which universality holds?
Is the rate in \cref{cor:higher-moment-rate} optimal, or can an online algorithm
respond more efficiently to rare large entries?  Further algorithmic
questions are how to compute a near-optimal Brownian control
efficiently and what numerical precision its implementation requires.
Computing $\ell_\star$ would also give a numerical value for the universal leading constant
in the threshold as $\kappa\downarrow0$.

\section*{AI Usage Disclosure}

OpenAI's GPT-5.6 Sol was used for exploratory discussions during the
development of the proofs and to assist with drafting and revising
the manuscript. The authors take full responsibility for the paper's
contents and correctness.

\appendix

\section{Projection and approximation of Brownian controls}
\label[appendix]{app:controls}

\subsection{Projection onto the Brownian filtration}
\label{app:projection}

\begin{lemma}
\label{lem:Brownian-projection}
Let $B$ be a Brownian motion in a filtration $\cG=(\cG_t)_{0\le t\le1}$, and
let $a$ be progressively measurable with respect to $\cG$, with
$\E\int_0^1a_t^2\,\dd t<\infty$.  Let $C\subset\R$ be closed and convex and
suppose
\begin{equation}
B_1+\int_0^1a_t\,\dd t\in C\qquad\text{a.s.}
\label{eq:projection-hypothesis}
\end{equation}
Let $u$ be the orthogonal projection of $a$ in
$L^2([0,1]\times\Omega,\dd t\otimes\Pp)$ onto the closed subspace of
processes predictable with respect to $\cF^B$.  Since the Brownian filtration is
continuous, this projection satisfies
$u_t=\E[a_t\mid\cF_t^B]$ for almost every $t$.  Then
\begin{equation}
B_1+\int_0^1u_t\,\dd t
=\E\left[B_1+\int_0^1a_t\,\dd t\,\middle|\,\cF_1^B\right]
\in C\qquad\text{a.s.},
\label{eq:projection-terminal}
\end{equation}
and
\begin{equation}
\E u_t^2\le\E a_t^2
\qquad\text{for a.e.\ }t.
\label{eq:projection-energy}
\end{equation}
\end{lemma}

\begin{proof}
Because $B$ is Brownian in $\cG$, for every integrable $\cG_t$-measurable
random variable $H$,
\begin{equation}
\E[H\mid\cF_1^B]=\E[H\mid\cF_t^B].
\label{eq:immersion-identity}
\end{equation}
Apply this to bounded truncations of $a_t$ and use conditional Fubini and
$L^1$ convergence.  This gives the equality in
\eqref{eq:projection-terminal}.  Conditional expectation preserves a closed
convex subset of the real line, which gives the last assertion of
\eqref{eq:projection-terminal}.  The energy bound is conditional Jensen.
\end{proof}

\subsection{Bounded approximations}
\label{app:bounded-approximations}

We prove \cref{prop:FJLNW-approximants}.  The bounded approximation
result of \cite[Lemma~4.15]{FiedlerJacksonLackerNilesWeed26}, combined
with its diffusion representation in Lemma~2.4 and Brownian scaling,
gives a probability law $m_T$ supported on $[-r_{c,T},r_{c,T}]$ and
processes
\[
X_t^{(n)}=B_t^{(n)}+\int_0^t a_r^{(n)}\,\dd r,
\qquad 0\le t\le T,
\]
with the following properties.  For each $n$, the process $B^{(n)}$ is
Brownian in a filtration to which the bounded control $a^{(n)}$ is
progressively measurable, and
\[
\|a_t^{(n)}\|_{L^2}\le c\quad\text{for a.e. }t,
\qquad
\mathbf d_p\bigl(\Law(X_T^{(n)}),m_T\bigr)\longrightarrow0.
\]
The bounded approximation argument holds on every fixed finite horizon
(Section~5 of the cited paper), including the horizon $(c/\cstar)^2T$
obtained by scaling the budget to $\cstar$.  Rescaling a horizon-$T$
control to unit time multiplies its drift budget by $\sqrt T$, which gives
$r_{c,T}=\sqrt T\,\mathcal R(c\sqrt T)$.

Project $a^{(n)}$ onto the completed natural filtration of $B^{(n)}$
and call the projected control $u^{(n)}$.  The projection identity used
in \cref{lem:Brownian-projection}, now on $[0,T]$, gives
\[
Y_T^{(n)}:=B_T^{(n)}+\int_0^T u_t^{(n)}\,\dd t
=\E[X_T^{(n)}\mid\mathcal F_T^{B^{(n)}}].
\]
The projected control remains bounded and satisfies the same pointwise
$L^2$ budget.  Since $x\mapsto(|x|-r_{c,T})_+^p$ is convex,
conditional Jensen and an optimal coupling with $m_T$ imply
\[
\E(|Y_T^{(n)}|-r_{c,T})_+^p
\le\E(|X_T^{(n)}|-r_{c,T})_+^p
\le\mathbf d_p(\Law(X_T^{(n)}),m_T)^p\longrightarrow0.
\]
Each control adapted to the natural Brownian filtration can be represented
on canonical Wiener space, with the fixed Brownian motion $B$ in the
statement.  The bound on the distance to $[-r_{c,T},r_{c,T}]$ proves
\eqref{eq:bounded-deficit}.
\qed

\subsection{Approximation by simple predictable controls}
\label{app:finite-observation}

We prove \cref{prop:finite-observation}.  Consider the dyadic partition
$0=t_{0,n}<\cdots<t_{J_n,n}=T$.  Denote
$\Delta_{j,n}=t_{j+1,n}-t_{j,n}$.  The orthogonal projection of $u$ onto the
space of predictable processes which are constant on each
$(t_{j,n},t_{j+1,n}]$ is
\begin{equation}
u^{(n)}_t=\zeta_{j,n},
\qquad
\zeta_{j,n}:=
\E\left[
\frac1{\Delta_{j,n}}\int_{t_{j,n}}^{t_{j+1,n}}u_r\,\dd r
\middle|\cF_{t_{j,n}}^B
\right].
\label{eq:predictable-projection}
\end{equation}
Then $u^{(n)}\to u$ in $L^2([0,T]\times\Omega,\dd t\otimes\Pp)$.
Conditional expectation and Minkowski's
inequality give
\begin{equation}
\norm{\zeta_{j,n}}_{L^2}
\le\frac1{\Delta_{j,n}}
\int_{t_{j,n}}^{t_{j+1,n}}\norm{u_r}_{L^2}\,\dd r
\le c_0.
\label{eq:block-L2-cap}
\end{equation}
Choose a rational $K_u>\norm{u}_{L^\infty}$; then
$|\zeta_{j,n}|\le K_u$.  Progressively measurable controls have predictable
representatives modulo $\dd t\otimes\Pp$-null sets.

Fix $n$.  The Brownian sigma-field $\cF_{t_{j,n}}^B$ is generated, modulo
completion, by evaluations at rational times not exceeding $t_{j,n}$.
Consequently, bounded continuous functions of finitely many such evaluations
are dense in $L^2(\cF_{t_{j,n}}^B)$.  Choose
\begin{equation}
\widetilde\zeta_{j,n}
=F_{j,n}(B_{r_{j,n,1}},\ldots,B_{r_{j,n,m_{j,n}}})
\label{eq:cylinder-approximation}
\end{equation}
with $r_{j,n,\ell}\le t_{j,n}$ and with $F_{j,n}$ bounded and globally
Lipschitz, so that the $L^2$ error is arbitrarily small.  Clipping the
function at $[-K_u,K_u]$ does not enlarge the error.  A continuous function
on a large rational box can be approximated uniformly by a rational
piecewise-affine function; tapering outside the box preserves boundedness
and global Lipschitz continuity.  Hence the approximants may be given
finite rational descriptions.

Let $\beta^{(n)}$ be the step process with block values
$\widetilde\zeta_{j,n}$.  First choose the cylinder approximations so that
\begin{equation}
\max_j\norm{\widetilde\zeta_{j,n}-\zeta_{j,n}}_{L^2}
<c_1-c_0.
\label{eq:block-slack-preserved}
\end{equation}
Then the $L^2$ norm of $\beta^{(n)}_t$ is below $c_1$ uniformly in $t$.
Next choose them so that $\beta^{(n)}-u^{(n)}$ is small in
$L^2([0,T]\times\Omega)$ and send
$n\to\infty$.  We obtain
\begin{equation}
\norm{\beta^{(n)}-u}_{L^2([0,T]\times\Omega)}\longrightarrow0.
\label{eq:beta-u-L2}
\end{equation}
The processes are uniformly bounded.  For $p>2$, interpolation between
\eqref{eq:beta-u-L2} and the uniform $L^\infty$ bound gives convergence in
$L^p([0,T]\times\Omega)$; for $1\le p\le2$, it follows directly from
H\"older's inequality.  Minkowski therefore gives
\begin{equation}
\left\|\int_0^T(\beta_t^{(n)}-u_t)\,\dd t\right\|_{L^p}
\le\int_0^T\norm{\beta_t^{(n)}-u_t}_{L^p}\,\dd t
\longrightarrow0.
\label{eq:integrated-control-error}
\end{equation}
Taking $n$ large proves \cref{prop:finite-observation}.
\qed

\section{Measurability of the coupling}
\label[appendix]{app:coupling-kernels}

We construct the measurable kernels used in
\cref{lem:product-innovation}.  Let $\Theta$ be a standard Borel parameter
space, let $\mu$ be a probability measure on $\R$, and let
$g:\Theta\times\R\to[0,\infty)$ be jointly measurable with
$\int g(\vartheta,x)\,\mu(\dd x)=1$.
The first variable $X$, with law $g(\vartheta,\cdot)\mu$, has already
been sampled.  We attach a variable $Y$ with law $\mu$ without changing $X$.

When $\int(1-g(\vartheta,z))_+\,\mu(\dd z)>0$, define the residual law
\begin{equation}
\lambda^-_\vartheta(\dd y)
:=\frac{(1-g(\vartheta,y))_+\,\mu(\dd y)}
        {\int(1-g(\vartheta,z))_+\,\mu(\dd z)}.
\label{eq:residual-kernels}
\end{equation}
For $g(\vartheta,x)>0$, use the conditional kernel
\[
K_\vartheta(x,\dd y)
=\frac{\min\{g(\vartheta,x),1\}}{g(\vartheta,x)}\delta_x(\dd y)
 +\frac{(g(\vartheta,x)-1)_+}{g(\vartheta,x)}
  \lambda^-_\vartheta(\dd y).
\]
Thus, given $X=x$, keep $Y=x$ with probability
$\min\{1,1/g(\vartheta,x)\}$; otherwise draw $Y$ from the residual law.
If the residual mass is zero, then $g(\vartheta,\cdot)=1$ $\mu$-almost
everywhere, and we use the diagonal kernel.  Where $g(\vartheta,x)=0$,
use any fixed probability kernel; this is a null set for the law of $X$.
All kernels are jointly measurable and can be sampled using independent
uniform random variables.

Since $\int(g-1)_+\,\dd\mu=\int(1-g)_+\,\dd\mu$, integrating the kernel
against $g\,\dd\mu$ gives the marginal law $\mu$ for $Y$.
The common part has mass $\int\min\{g,1\}\,\dd\mu$, so the coupling is
maximal.  For every nonnegative measurable $F$,
\begin{equation}
\E[F(X)\1_{\{X\ne Y\}}]
=\int F(x)(g(\vartheta,x)-1)_+\,\mu(\dd x),
\label{eq:residual-first-moment}
\end{equation}
and the analogous identity with $(1-g)_+$ holds for $Y$.  Consequently,
\begin{equation}
\E|X-Y|^2
\le2\int x^2|g(\vartheta,x)-1|\,\mu(\dd x)
\le2\int(1+x^2)|g(\vartheta,x)-1|\,\mu(\dd x).
\label{eq:maximal-coupling-moment}
\end{equation}

Apply this kernel successively with $\mu=\nu$ and $g=g_i$ from
\eqref{eq:gi-density}, conditionally on $(\eps,S)$, using independent
uniforms.  Each $Z_i$ is sampled from the observed prefix
$(\zeta_1,\ldots,\zeta_i)$ and its auxiliary uniform, so conditioning on
the preceding $Z$ coordinates does not change the conditional density
of $\zeta_i$.  The conditional law of $Z_i$ given $(\eps,S)$ and the
preceding $\zeta$ and $Z$ coordinates is $\nu$.  Induction gives
$Z\sim\nu^{\otimes M}$, independent of $(\eps,S)$, while preserving the
original online trajectory.  In the temporal iteration $\nu=\nu_N$,
and the conditional joint law in \eqref{eq:product-innovations} is the
same at every step.

\section{F\"ollmer drifts and Euler approximation}
\label[appendix]{app:Follmer}

The score representation and the construction of terminal laws with
shrinking support are
from \cite[Lemmas~4.1--4.2]{FiedlerJacksonLackerNilesWeed26}, in the
formulation for compactly supported densities in
\cite[Lemmas~C.1--C.3]{JoLee26}.
The Euler estimate follows \cite[Lemma~4.5 and Appendix~C.4]{JoLee26}.
We include the proofs to track the reduced drift budget, the
deterministic terminal cutoff, and the resulting bounds uniformly over
all coordinates.

\subsection{Localization and the score martingale}
\label{app:score-martingale}

Fix $p\ge2$.  Let $f\ge0$ satisfy $\int f\,\dd\Gamma_h=1$.
Assume that $f$ is compactly supported, absolutely continuous, and vanishes
at the boundary of its support.  Define the logarithmic derivative of the
density $f$ relative to $\Gamma_h$ by
\begin{equation}
s_f(x):=\begin{cases}f'(x)/f(x),&f(x)>0,\\0,&f(x)=0,
\end{cases}
\label{eq:score-definition}
\end{equation}
and assume $s_f\in L^p(f\Gamma_h)$.  Put
\begin{equation}
q_t(x):=P_{h-t}f(x),
\qquad
v(t,x):=\partial_x\log q_t(x),
\qquad 0\le t<h.
\label{eq:q-v-Follmer}
\end{equation}
On canonical Brownian space define $\mathbb Q$ by
\begin{equation}
\frac{\dd\mathbb Q}{\dd\Pp}=f(B_h).
\label{eq:Follmer-change-measure}
\end{equation}
Then $q_t(B_t)=\E[f(B_h)\mid\cF_t^B]$ is the density process.  Localizing away
from zero and infinity, It\^o's formula and the backward heat equation give
\begin{equation}
\dd q_t(B_t)=q_t(B_t)v(t,B_t)\,\dd B_t.
\label{eq:density-SDE}
\end{equation}
Girsanov's theorem therefore makes
\begin{equation}
\widetilde B_t:=B_t-\int_0^t v(s,B_s)\,\dd s
\label{eq:Q-Brownian}
\end{equation}
a Brownian motion under $\mathbb Q$.  Differentiation under convolution and
Bayes' formula yield
\begin{equation}
v(t,B_t)
=\frac{P_{h-t}f'(B_t)}{P_{h-t}f(B_t)}
=\E_{\mathbb Q}[s_f(B_h)\mid\cF_t^B].
\label{eq:score-martingale}
\end{equation}
Conditional Jensen gives
\begin{equation}
\sup_{t<h}\norm{v(t,B_t)}_{L^p(\mathbb Q)}
\le\norm{s_f}_{L^p(f\Gamma_h)}.
\label{eq:score-martingale-Lp}
\end{equation}
The terminal law under $\mathbb Q$ is $f\Gamma_h$.  The bound
\eqref{eq:score-martingale-Lp} implies
$\E_{\mathbb Q}\int_0^h|v(t,B_t)|^2\,\dd t<\infty$ when $p\ge2$.
Consequently the drift integral has a continuous limit at $h$, and the
localized Girsanov identities extend first to every interval $[0,h-\epsilon]$
and then to $[0,h]$ by continuity.  This proves the weak diffusion
representation up to and including the terminal time.

For the even log-concave densities below,
\eqref{eq:Follmer-derivative} makes the drift globally Lipschitz with linear
growth on each $[0,h-\epsilon]$.  Thus the SDE has a unique strong solution
there.  Uniqueness in law identifies it with the preceding weak realization.
These solutions agree on overlapping intervals.  Their common law and
the integrability of the drift give a continuous extension to $h$,
whose law is $f\Gamma_h$.  This gives the strong solution asserted in
\cref{prop:compact-confinement} up to and including time $h$.

\subsection{Construction of the compactly supported density}
\label{app:tent}

Set $\varpi:=c_F/(4p)$.  Define
\begin{equation}
r_h:=2\sqrt{p h\log(4\varpi^{-2}/h)}
\label{eq:tent-radius}
\end{equation}
for small $h$, let $H_h$ equal one on $[-r_h/2,r_h/2]$, decrease linearly to
zero on the two adjacent half-intervals, and vanish outside $[-r_h,r_h]$,
and set
\begin{equation}
f_h:=\frac{H_h^p}{\int H_h^p\,\dd\Gamma_h},
\qquad s_p(h,c_F):=r_h.
\label{eq:tent-density}
\end{equation}
The Gaussian tail bound gives
\begin{equation}
\Gamma_h([r_h/2,\infty))
\le e^{-r_h^2/(8h)}
=\left(\frac \varpi2\right)^p h^{p/2}.
\label{eq:tent-tail}
\end{equation}
Put $\ell_h=\log(4\varpi^{-2}/h)$.  The two sloping pieces of the cutoff
have $|H_h'|=2/r_h$, so \eqref{eq:tent-tail} gives
\[
\int|H_h'|^p\,\dd\Gamma_h
\le 2(2/r_h)^p e^{-r_h^2/(8h)}
=\frac{2^{1-p}\varpi^p}{(p\ell_h)^{p/2}}.
\]
Also
$\int H_h^p\,\dd\Gamma_h\ge1-2e^{-r_h^2/(8h)}\ge1/2$
for all sufficiently small $h$.  Taking $h$ still smaller so that
$p\ell_h\ge1$ proves
\begin{equation}
\norm{H_h'}_{L^p(\Gamma_h)}
\le \varpi\norm{H_h}_{L^p(\Gamma_h)}.
\label{eq:tent-gradient}
\end{equation}
The logarithmic derivative of $f_h$ relative to $\Gamma_h$ is $pH_h'/H_h$ where $H_h>0$, so
\begin{equation}
\int|s_{f_h}|^pf_h\,\dd\Gamma_h
=p^p\frac{\int|H_h'|^p\,\dd\Gamma_h}
             {\int H_h^p\,\dd\Gamma_h}
\le(c_F/4)^p.
\label{eq:tent-score}
\end{equation}
This proves \eqref{eq:Follmer-support}--\eqref{eq:Follmer-score-budget}.
The radius tends to zero because $h\log(1/h)\to0$.

Gaussian convolution preserves log-concavity \cite{Prekopa71}, hence
$\partial_{xx}\log P_{h-t}f_h\le0$.  For $\tau=h-t$, let $Y$ have
Lebesgue density proportional to $f_h(y)e^{-(y-x)^2/(2\tau)}$.
Direct differentiation gives
\begin{equation}
\partial_{xx}\log P_\tau f_h(x)
=-\frac1\tau+
\frac{\Var(Y)}{\tau^2}
\ge-\frac1\tau.
\label{eq:log-heat-Hessian}
\end{equation}
This proves \eqref{eq:Follmer-derivative}.  Since $f_h$ is even,
$v_h(t,0)=0$, and therefore
\begin{equation}
|v_h(t,x)|\le\frac{|x|}{h-t}.
\label{eq:Follmer-linear-growth}
\end{equation}

\subsection{Euler approximation}
\label{app:Euler}

We drive the diffusion and its Euler approximation with the same Brownian path, as in
\cite[Appendix~C.4]{JoLee26}, and keep the cutoff index explicit.

Let $Z$ be the strong solution of \eqref{eq:Follmer-SDE} on horizon
$h_N$, driven by $B$.  All expectations in this subsection refer to this
realization.  Set
$\mathcal V_t:=v_{h_N}(t,Z_t)$ for $t<h_N$, with terminal value
$\mathcal V_{h_N}:=s_{f_{h_N}}(Z_{h_N})$.  By
\eqref{eq:score-martingale} and uniqueness in law, $\mathcal V$ is an
$L^p$-bounded martingale in the filtration of $B$,
and $\mathcal V_0=0$ by evenness.  The backward heat equation implies
$\partial_t v+\tfrac12\partial_{xx}v+v\partial_xv=0$ for $t<h_N$.
It\^o's formula along the strong solution therefore gives the explicit
representation
\begin{equation}
\mathcal V_t=\int_0^t\zeta_r\,\dd B_r,
\qquad \zeta_r=\partial_xv_{h_N}(r,Z_r).
\label{eq:M-representation}
\end{equation}
The terminal $L^p$ bound and the reverse Burkholder--Davis--Gundy
inequality give
\[
\left\|\int_0^{h_N}|\zeta_r|^2\,\dd r\right\|_{L^{p/2}}^{1/2}
\le C_p\|\mathcal V_{h_N}\|_{L^p}.
\]
Thus the representation extends to $h_N$ in $L^p$ in the Brownian
filtration of the diffusion.
For a mesh interval $[j\Delta,(j+1)\Delta]$, $\Delta=1/N$, define
\begin{align}
d_j&:=\int_{j\Delta}^{(j+1)\Delta}(\mathcal V_t-\mathcal V_{j\Delta})\,\dd t
\nonumber\\
&=\int_{j\Delta}^{(j+1)\Delta}
((j+1)\Delta-r)\zeta_r\,\dd B_r.
\label{eq:Euler-local-martingale}
\end{align}
Let $\widehat Z$ be the Euler scheme with the truncated drift and the
same Brownian increments, and put
$e_j=\widehat Z_j-\bar Z_{j\Delta}$.  For $0\le j<j_N^*$ the entire
update interval precedes the cutoff, and
$\bar Z_t=Z_t$ throughout that interval.  The mean-value theorem and
\eqref{eq:Follmer-derivative} therefore give
\begin{equation}
e_{j+1}=\vartheta_je_j-d_j,
\qquad 0\le\vartheta_j\le1,
\qquad 0\le j<j_N^*.
\label{eq:Euler-error-recursion}
\end{equation}
These updates determine the error through mesh time $j_N^*/N$.
For $j_N^*\le j<L_N$, both truncated processes have zero drift, so
$e_{j+1}=e_j$.  The partial sums of $d_j$, stopped before $j_N^*$,
are mesh values of the martingale
\begin{equation}
\int_0^t w_N(r)\zeta_r\,\dd B_r.
\label{eq:weighted-error-martingale}
\end{equation}
The deterministic weight is
\[
w_N(r)=
\begin{cases}
(j+1)\Delta-r,&j\Delta\le r<(j+1)\Delta,\quad 0\le j<j_N^*,\\
0,&r\ge j_N^*\Delta.
\end{cases}
\]
In particular, $0\le w_N\le\Delta$.
By \cref{lem:dissipative-recursion}, Doob's inequality, the
Burkholder--Davis--Gundy inequality, and its reverse form
\cite[Chapter~IV, Section~4]{RevuzYor99},
\begin{equation}
\left\|\max_j|e_j|\right\|_{L^p}
\le C_p\Delta\norm{\mathcal V_{h_N}}_{L^p}
\le\frac{C_pc_F}{4N}.
\label{eq:Euler-error-final}
\end{equation}
The error is constant after the cutoff, so this bound holds to the
terminal time.  Taking an $\ell^p$ bound over the $M$ independent rows
proves \eqref{eq:Euler-all-row}.

Before the cutoff, \eqref{eq:Follmer-derivative} gives
\begin{equation}
|v_{h_N}(j/N,\widehat Z_j)-v_{h_N}(j/N,Z_{j/N})|
\le\delta_N^{-1}|e_j|,
\label{eq:Euler-drift-comparison}
\end{equation}
which proves \eqref{eq:Euler-request-error}; the same estimate applies
at $j=j_N^*$ because $\bar Z_{j/N}=Z_{j/N}$ there.
For $j<j_N^*$, \eqref{eq:Euler-error-final} gives
\[
\|\bar v_N(j/N,\widehat Z_j)\|_{L^p}
\le\frac{c_F}{4}\left(1+\frac{C_p}{N\delta_N}\right)
\le\frac{c_F}{2}
\]
for all sufficiently large $N$, since $N\delta_N=N^{23/24}\to\infty$.
For $j\ge j_N^*$ the truncated feedback is zero, proving
\eqref{eq:Euler-request-Lp}.

Finally,
\begin{equation}
\dist(\widehat Z_{L_N},[-s_p(h_N,c_F),s_p(h_N,c_F)])
\le |\widehat Z_{L_N}-\bar Z_{h_N}|
    +|\bar Z_{h_N}-Z_{h_N}|.
\label{eq:terminal-distance-decomposition}
\end{equation}
The first term has $L^p$ norm $O(N^{-1})$ by
\eqref{eq:Euler-error-final}; the second has norm $O(\delta_N+N^{-1})$ by
\eqref{eq:truncated-full-Lp}.  A union bound over $M=O(N)$ rows and
Markov's inequality prove \eqref{eq:Euler-terminal-tail}.
\qed

\section{Quantitative estimates for the online algorithm}
\label[appendix]{app:quantitative}

This appendix gives the parameter choices and probability estimates for
\eqref{eq:good-event-tail}, and proves that the basic comparison events
exclude every failure of an admissibility test.

\subsection{Choice of parameters}
\label{app:parameters}

Fix $\theta>0$ and choose $s,p$ as in \eqref{eq:moment-choices}.
The cutoff tolerance is $\delta_N=N^{-1/24}$.  By
\eqref{eq:iterated-martingale} and
\eqref{eq:all-row-strong-approximation}, both coupling errors have
$L^s$ norm $O(N^{-(s-4)/(6s)})$.  Markov's inequality therefore gives,
for every fixed $c>0$,
\begin{equation}
\Pp(\text{either coupling error}>c\delta_N)
\le C_cN^{-(s/8-2/3)}=O(N^{-\theta}),
\label{eq:coupling-tail-exponent}
\end{equation}
since $s>8\theta+16/3$.  The predictable bias is $o(\delta_N)$ because
$s>12$ implies $1/2-2/s>1/24$.

The conditions on $p$ give
\[
\frac p{24}-1>\theta,\qquad
\frac p{2s}-1>\theta,\qquad
\frac{23p}{24}-1>\theta.
\]
Indeed, $p>2s(1+\theta)$ and $s>12$ imply the first two inequalities,
and the third follows from the first.  These are the exponents for
terminal confinement, empirical drift moments, and Euler drift comparison.
Finally, $s>2/\theta$ makes the expected-excess bound vanish.
All these choices are available under a fourth-moment bound on the
original entries, since higher input moments are used only after
truncation.  To obtain $O(N^{-A})$ bounds in
\cref{prop:truncated-polynomial}, take $\theta=2A+3$.

\subsection{Empirical moment bounds}
\label{app:empirical-moments}

We repeatedly use the following consequence of Rosenthal's inequality
\cite{Rosenthal70}.

\begin{lemma}
\label{lem:empirical-moments}
Let $A_{j,i}$, $1\le j\le J_N$, $1\le i\le M$, be such that for each fixed
$j$ the variables $(A_{j,i})_{i\le M}$ are independent and identically
distributed, and suppose
\begin{equation}
J_N\le CN,
\qquad
M\asymp N,
\qquad
\sup_{N,j}\E|A_{j,1}|^p\le C_p
\label{eq:empirical-hypotheses}
\end{equation}
for some $p\ge2s$.  Then, for each fixed $\eps>0$,
\begin{equation}
\Pp\left(
\max_{j\le J_N}\left|
\frac1M\sum_{i=1}^M|A_{j,i}|^s-
\E|A_{j,1}|^s
\right|>\eps
\right)
\le C_{p,s,\eps}J_N M^{-p/(2s)}
\le C'_{p,s,\eps}N^{1-p/(2s)}.
\label{eq:empirical-moment-tail}
\end{equation}
The same estimate holds with $s$ replaced by any $1\le r\le s$.
\end{lemma}

\begin{proof}
Apply Rosenthal's inequality with exponent $p/s\ge2$ to the centered
variables $|A_{j,i}|^s-\E|A_{j,i}|^s$.  It gives
\begin{equation}
\E\left|\frac1M\sum_{i=1}^M
(|A_{j,i}|^s-\E|A_{j,i}|^s)\right|^{p/s}
\le C M^{-p/(2s)}.
\label{eq:Rosenthal-mean}
\end{equation}
Markov's inequality and a union bound give the first estimate.
The assumptions $J_N=O(N)$ and $M\asymp N$ give the second.
\end{proof}

For the drifts evaluated along the ideal Euler schemes,
\eqref{eq:Euler-request-Lp} supplies
\eqref{eq:empirical-hypotheses}.  By \eqref{eq:moment-choices}, the exponent in
\eqref{eq:empirical-moment-tail} is smaller than $-\theta$.  The pre-split drift is bounded and piecewise defined on finitely many
blocks, so its empirical moments have exponential tails.  The ideal terminal variable $Y_T^\beta$ is
a Gaussian variable plus a bounded drift, hence has moments of all orders.
Applying the same calculation at the final time gives the polynomial
probability bound in \eqref{eq:empirical-deficit}.

\subsection{Discretization of the observation times}
\label{app:rounding}

Let $\mathcal T$ be the finite set of observation and block times in
\cref{def:finite-observation}.  Round each $t\in\mathcal T$ down to
$\lfloor Nt\rfloor/N$.  For independent Brownian motions,
\begin{equation}
\max_{i\le M,\,t\in\mathcal T}
|B_i(t)-B_i(\lfloor Nt\rfloor/N)|
=O_{\Pp}\left(\sqrt{\frac{\log N}{N}}\right).
\label{eq:Brownian-rounding}
\end{equation}
Moreover, since $\delta_N=N^{-1/24}$,
\begin{equation}
\Pp\left(
\max_{i,\,t\in\mathcal T}|B_i(t)-B_i(\lfloor Nt\rfloor/N)|
>c\delta_N
\right)
\le C e^{-c'N^{11/12}}.
\label{eq:Brownian-rounding-tail}
\end{equation}
These are immediate from the Gaussian tail bound and a union bound over
$O(M)$ increments.

\subsection{Admissibility and terminal integrability}
\label{app:stopping}

Assume $N\ge N_0$ and $M/N\in I_\alpha$, as holds eventually along
every sequence under consideration.  Let $\tau$ be the first failed
test of the candidate drift
$b_k+\gamma\eta_k$, and set $\tau=N+1$ if there is no failure.  At a
failure the applied drift is replaced immediately by $\gamma\eta_k$,
and all later steps use this auxiliary drift.  Hence the applied
drifts satisfy \eqref{eq:drift-cap}--\eqref{eq:drift-delocalization}
on every trajectory.  Thus \cref{prop:iterated-coupling} applies
globally, including at and after a failed test.  Only comparisons with the intended drifts are restricted to times before $\tau$.

Define $\cG_N$ as the intersection of the following comparison events:
the two global
coupling errors are at most a sufficiently small fixed multiple of
$\delta_N$; the finitely many Brownian rounding errors and the required
partial perturbation sums have the same bound; the ideal pre-split
empirical norms and the empirical moment of $d_a(Y_{T,i}^\beta)$ have their
prescribed strict slack; all empirical norms of the ideal Euler drifts satisfy their
bounds; the terminal confinement errors lie inside the unused radius;
and the complement of the small-norm event in
\eqref{eq:perturbation-norm-event} holds
at every candidate drift, including the first failed test if there is
one.  The norm estimate holds uniformly over all base drifts.
Its failure probability is exponentially
small by conditioning at each step on the causal past, since $b_k$ is
computed before $\eta_k$ is drawn.

We prove that all admissibility tests pass on $\cG_N$ by excluding a
first failure.  Suppose first that
$\tau\le k_0$.  Every observation used to form $b_\tau$ has index at most
$\tau-1$ by \eqref{eq:causal-mesh-control}.  The comparison between $W$ and $B$
and the Lipschitz bounds for $F_j$ therefore compare the candidate drift
$b_\tau$ with the ideal Brownian drift before the test at that step.
Its empirical $L^2$ and $L^s$ norms lie inside the allowed
bounds.  The triangle inequality controls the energy after perturbation,
and the perturbation lower-norm bound controls its normalized maximum coordinate.
The choice of $\Lambda$ and the fixed slack imply that all three tests
pass at step $\tau$, a contradiction.

If $k_0<\tau\le N$, then no test has failed through time $k_0$, so the
split-time comparison gives \eqref{eq:correction-Ls}.  Put $j=\tau-k_0-1$.
All updates in the final block producing $U_j$ occurred before $\tau$.
The comparison of the nondecreasing, one-Lipschitz Euler updates bounds
$\norm{U_j-\widehat Z_j}_\infty$ by a small multiple of $\delta_N$.
Before the cutoff, the feedback Lipschitz constant is at most
$\delta_N^{-1}$; after the cutoff both feedbacks are zero.  Thus the
feedback at step $\tau$ differs by at most the fixed drift
slack from the ideal feedback.  Combining the ideal empirical norms with
\eqref{eq:correction-Ls} gives the required $L^2$ and $L^s$
bounds.  The same perturbation argument makes the magnitude and
delocalization tests pass.  This is again a contradiction.  Consequently
$\cG_N\subset\{\tau>N\}$, and the terminal identity applies on $\cG_N$.

We estimate each defining failure event separately.
The Burkholder--Rosenthal inequality in the causal filtration gives
\begin{equation}
\Pp\left(\max_k\norm{Q_k}_\infty>c\delta_N\right)
\le C N^{-(s/8-2/3)}.
\label{eq:Q-tail-delta}
\end{equation}
The Gaussian strong approximation has the same exponent.  The
probability that the Brownian rounding errors or the partial perturbation
sums exceed a fixed multiple of $\delta_N$ is exponentially small,
because $\delta_N=N^{-1/24}$.  The ideal empirical and confinement
bounds follow from \cref{lem:empirical-moments} and
\eqref{eq:Euler-terminal-tail}--\eqref{eq:Euler-request-error}.
The ideal variables $Y_T^\beta$ have moments of every fixed order, so the
failure probability for the empirical moment of their distances to $[-a,a]$ can be made
$O(N^{-\theta})$ as well.  The union bound and the exponent choices give
$\Pp(\cG_N^c)=O(N^{-\theta})$.

Finally we verify the second-moment bound
\eqref{eq:global-L2-envelope}.  Because the applied drift has empirical
$L^s$ norm at most $\Lambda$, every coordinate satisfies
\begin{equation}
|r_{k,i}|\le \Lambda M^{1/s},
\label{eq:coordinate-request-envelope}
\end{equation}
and consequently
\begin{equation}
\max_i\left|\frac1N\sum_{k=1}^Nr_{k,i}\right|
\le \Lambda M^{1/s}.
\label{eq:drift-envelope}
\end{equation}
At time $N$, \eqref{eq:iterated-decomposition} and
\eqref{eq:all-row-strong-approximation} express $N^{-1/2}\sum_{k=1}^N S_kZ_k$ as a vector
of independent $N(0,1)$ variables plus an $L^{s}$ error of order
$N^{-(s-4)/(6s)}$.  This bound concerns the truncated discrepancy
$D^{\mathrm{tr}}_{M,N}$.  The Gaussian maximum has $L^2$ norm $O(\sqrt{\log M})$.
The bias and martingale terms in \eqref{eq:iterated-decomposition} have
bounded $L^2$ norm by \eqref{eq:iterated-bias} and
\eqref{eq:iterated-martingale}.  Combining these estimates proves
\eqref{eq:global-L2-envelope}.

\end{document}